\documentclass{article}

\usepackage[margin=1in,a4paper]{geometry}
\usepackage{amssymb,amsmath,,wasysym}
\usepackage{amsthm} 
\usepackage{xcolor}
\usepackage{float}
\usepackage[colorlinks,linkcolor=blue,citecolor=blue,pagebackref,hypertexnames=false, breaklinks]{hyperref}
\usepackage{circuitikz}

\newcommand{\brak}[1]{\left(#1\right)} 
\newcommand{\norm}[1]{{\left\lVert{#1}\right\rVert}}

\def\Xint#1{\mathchoice
    {\XXint\displaystyle\textstyle{#1}}%
    {\XXint\textstyle\scriptstyle{#1}}%
    {\XXint\scriptstyle\scriptscriptstyle{#1}}%
    {\XXint\scriptscriptstyle\scriptscriptstyle{#1}}%
    \!\int}
    \def\XXint#1#2#3{{\setbox0=\hbox{$#1{#2#3}{\int}$}
    \vcenter{\hbox{$#2#3$}}\kern-.5\wd0}}
    \def\fint{\Xint-}

\def\vint{\mathop{\mathchoice%
          {\setbox0\hbox{$\displaystyle\intop$}\kern 0.22\wd0%
           \vcenter{\hrule width 0.6\wd0}\kern -0.82\wd0}%
          {\setbox0\hbox{$\textstyle\intop$}\kern 0.2\wd0%
           \vcenter{\hrule width 0.6\wd0}\kern -0.8\wd0}%
          {\setbox0\hbox{$\scriptstyle\intop$}\kern 0.2\wd0%
           \vcenter{\hrule width 0.6\wd0}\kern -0.8\wd0}%
          {\setbox0\hbox{$\scriptscriptstyle\intop$}\kern 0.2\wd0%
           \vcenter{\hrule width 0.6\wd0}\kern -0.8\wd0}}%
          \mathopen{}\int}

\newcommand{\vast}{\bBigg@{4}}
\newcommand{\Vast}{\bBigg@{5}}
\makeatother
\newcommand{\eps}{\varepsilon}

\newcommand{\pip}{\varphi}

\newcommand{\supp}{{\rm supp}\, }

\def\esssup{\mathop\mathrm{\,ess\,sup\,}}

\theoremstyle{plain}
\newtheorem{thm}{Theorem}[section]
\newtheorem{theorem}{Theorem}[section]
\newtheorem{lem}[thm]{Lemma}
\newtheorem{cor}[thm]{Corollary}
\newtheorem{corollary}[thm]{Corollary}

\newtheorem{proposition}[thm]{Proposition}

\newtheorem{example}[thm]{Example}

\theoremstyle{definition}
\newtheorem{defn}[thm]{Definition}

\theoremstyle{remark}
\newtheorem{remark}[thm]{Remark}
\newcommand{\bremark}{\begin{remark} \em}
\newcommand{\eremark}{\end{remark} }
\usepackage{color}

\hbadness=\maxdimen

\usepackage[textwidth=3.2cm]{todonotes}

\begin{document}


\title{Korevaar-Schoen-Sobolev spaces and critical exponents in metric measure spaces}

\author{Fabrice Baudoin\footnote{Partly supported by the NSF grants DMS~1901315 \& DMS~2247117.}}

\maketitle

\begin{abstract}
We survey, unify and present new  developments in the theory of Korevaar-Schoen-Sobolev spaces on metric measure spaces. While this theory coincides with those of Cheeger and Shanmugalingam if the space is doubling and supports a Poincar\'e inequality, it offers new perspectives in the context of fractals  for which the approach by weak upper gradients is inadequate.
\end{abstract}

\tableofcontents 

\section{Introduction}

If $f:\mathbb R^n \to \mathbb{R}$ is a $C^1$ Lipschitz function then for every $x,y \in \mathbb{R}^n$
\[
| f(x) -f(y) | \le \| \nabla f \|_{\infty} \, \| x-y \|.
\]
One can rephrase this inequality as
\[
\sup_{r  >0 } \sup_{x,y, \| x- y\| \le r} \frac{|f(x)-f(y)|}{r} \le \liminf_{r \to 0 } \sup_{x,y, \| x- y\| \le r} \frac{|f(x)-f(y)|}{r}.
\]
More generally (as follows  for instance from \cite[Proposition 1.11]{MR1708448}) such an inequality still holds if $f$ is a Lipschitz function defined on a length metric space, i.e.  a space for which any pair of points $x,y$ can be connected with a rectifiable curve and the infimum of the length of such curves  is the distance $d(x,y)$.

\

In this work, we shall be interested in $L^p$ analogues in the context of a doubling metric measure space $(X,d,\mu)$: For $p \ge 1, r>0, \alpha>0$ and $f \in L^p(X,\mu)$, we define 
\[
E_{p,\alpha} (f,r) = \int_X \frac{1}{\mu(B(x,r))}\left( \int_{B(x,r)} \frac{|f(y)-f(x)|^p}{r^{p\alpha}} d\mu(y) \right)d\mu(x) 
\] 
and ask if for some $\alpha > 0$ there is a constant $C \ge 1$ (depending only on the geometry of the space $(X,d,\mu)$) such that
\begin{align}\label{question intro}
\sup_{r >0} E_{p,\alpha}(f,r) \le C \liminf_{r  \to 0} E_{p,\alpha}(f,r) 
\end{align}
holds for some non-constant functions $f$. An underlying insight is that if \eqref{question intro} holds for a large class of functions, then the underlying metric measure space has some level of  $L^p$ infinitesimal regularity and global controlled $L^p$ geometry.  In particular, the value $p=1$ is related to the existence of a rich theory of BV functions and sets of finite perimeter satisfying isoperimetric estimates, see \cite{BV2,BV3} and \cite{MR3558354}. The value $p=2$ is related to the existence of a nice Laplacian on the space, more precisely the existence of a  local Dirichlet form which can be constructed as a $\Gamma$-limit of the functionals $E_{2,\alpha}$ as $r\to 0$, see \cite{MR2161694} and \cite{MR1617040}.

In this  paper we will see that \eqref{question intro} holds in a large class of spaces of different nature. This class includes doubling spaces satisfying a Poincar\'e inequality but also some  fractals like the Vicsek set and the Sierpi\'nski  gasket. For some other spaces like the Sierpi\'nski  carpet the  validity of \eqref{question intro}  is still an open question\footnote{March 2024 update: The validity of \eqref{question intro} in the Sierpi\'nski  carpet  was recently proved for some $\alpha$ and some range of $p$'s by Yang in  \cite{yang2024korevaarschoen}  and then  shortly after,  for every $p >1$ by Murugan and Shimizu in \cite{murugan2023firstorder}.  }.

We will also show that if \eqref{question intro} holds,  a rich theory of Sobolev spaces develops  using the scale of the Korevaar-Schoen spaces first introduced in \cite{MR1266480} in a Riemannian setting. Assuming doubling and a $p$-Poincar\'e inequality, that theory is equivalent to the theory of Sobolev spaces built on the notion of weak upper gradients by Cheeger \cite{MR1708448} and Shanmugalingam \cite{Shanmu00} and also to the theory of Haj\l{}asz \cite{MR1401074}. However, on spaces like fractals where the set of rectifiable curves is not rich enough in the measure theoretic sense (see Remark \ref{comment N}) the theory built on weak upper gradients yields non-useful Sobolev spaces, often the whole $L^p$ space. By contrast the theory  we can develop using the Korevaar-Schoen spaces  still produces a fruitful set of  results which can be used to study further the geometry of the space; In particular a whole scale of Gagliardo-Nirenberg type Sobolev embeddings is  available. 

Furthermore, an appealing aspect of the theory of Korevaar-Schoen-Sobolev spaces is its close connection to the very rich theory of heat kernels and Dirichlet forms as was developed in \cite{BV1,BV2,BV3} after \cite{MR2075670}, \cite{MR3420351}, \cite{MR3484021}, and  \cite{MR2574998}.  Due to this connection, one can hope to export to a general metric measure space setting some  of the powerful heat kernel techniques, like the Bakry-\'Emery-Ledoux machinery, see \cite{BV2,BV3}.

\

The paper is organized as follows. Sections 2 and 3 are preliminary sections, we collect some useful and  mostly known results about the class of Besov-Lipschitz functions on a doubling metric measure space. In Section 4, we first introduce and study the Besov critical exponents of a metric measure space and discuss \eqref{question intro} in connection with the Korevaar-Schoen-Sobolev spaces which we define as the Besov-Lipschitz spaces at the critical exponent. Finally, we show how \eqref{question intro} yields Sobolev embeddings and Gagliardo-Nirenberg inequalities. In Section 5, we show after \cite{MR2415381} and \cite{MR1628655} that \eqref{question intro} holds with $\alpha=1$ if the space satisfies a $p$-Poincar\'e inequality and point out that the theory of Korevaar-Schoen-Sobolev spaces is then equivalent to Cheeger and Shanmugalingam theories. We also prove, and this is one of  our main contributions, that if the space satisfies a generalized $p$-Poincar\'e inequality  and a controlled cutoff condition similar to that of \cite{MR2228569}, then \eqref{question intro} holds with a parameter $\alpha$ possibly greater than one. The case $p=2$ in Dirichlet spaces with sub-Gaussian heat kernel estimates is then discussed as a corollary of this general approach.
 
In Section 6, we discuss in  detail the Korevaar-Schoen-Sobolev spaces in two popular examples of fractals: The Vicsek set and the Sierpi\'nski  gasket. We show that those two examples satisfy for every $p \ge 1$ the inequality \eqref{question intro} for some value $\alpha=\alpha_p >1$ and as a consequence obtain new Nash inequalities. Finally, in Section 7,  we review some of the results in \cite{BV2,BV3,BV1} about the connection between the Dirichlet forms theory and the Korevaar-Schoen-Sobolev spaces.  Throughout the text several open problems and possible research directions are discussed.

\

\textbf{Acknowledgments:}  
\begin{itemize}
\item The author would like to thank Patricia Alonso-Ruiz, Li Chen, and Nageswari Shanmugalingam for stimulating discussions on  topics related to this work  during a workshop at the University of Texas A\&M in August 2022 and  also thank Takashi Kumagai for relevant comments on a very early version of the draft.
\item The author also thanks an anonymous referee for a meticulous reading which improved the presentation of the paper.
\end{itemize}

\section{Setup}

Our setting is that of \cite{HKST15}. Throughout the paper, let $(X,d,\mu)$ be a metric measure space\footnote{We assume that $X$ has more than one point so that there exist non-constant functions and $\mathrm{diam}(X) >0$.} where  $\mu$ is a Borel regular measure.  Open metric balls will be denoted by
\[
B(x,r)= \{ y \in X : d(x,y)<r \}.
\] 
Sometimes, when convenient, if $B$ is a ball and $\lambda>0$ we will denote by $\lambda B$ the ball with same center and radius multiplied by $\lambda$.

We will always assume that the measure $\mu$ is doubling  and positive in the sense that there exists a constant $C>0$, called the doubling constant, such that for every $x \in X, r>0$,
\[
0< \mu (B(x,2r)) \le C \mu(B(x,r)) <+\infty.
\]

It follows from the doubling property of $\mu$ (see \cite[Lemma 8.1.13]{HKST15})  that there is a constant $0<Q\ <\infty$ and $C > 0$ such that 
whenever $0<r\le R$ and $x\in X$, we have 
\begin{equation}\label{eq:mass-bounds}
 \frac{\mu(B(x,R))}{\mu(B(x,r))}
 \le C\left(\frac{R}{r}\right)^Q.
\end{equation}

Another well-known consequence of the doubling property is the availability of maximally separated $\eps$-coverings with the bounded overlap property and subordinated Lipschitz partitions of unity (see \cite[Pages 102-104]{HKST15} or \cite[Appendix B.7]{MR1699320}):
\begin{proposition}[Controlled Lipschitz partition of unity]\label{partition of unity}
Let $\varepsilon >0$. There exists a countable subset $A(\varepsilon)$ of X such that:
\begin{itemize}
\item $d(a_1,a_2) \ge \varepsilon$ for all $a_1,a_2\in A(\varepsilon)$  with $a_1\neq a_2$;
\item $X = \bigcup_{ a \in A(\varepsilon) } B(a,\varepsilon)$.
\end{itemize}
Moreover, for any $k> 0$ there exists a constant $\beta (k) > 0$ depending only on $k$ and on the doubling constant such that:
\begin{itemize}
\item $\sum_{a\in A(\varepsilon)} 1_{ B(a,k\varepsilon)} (x)\le \beta(k)$ for every $x \in X$.
\end{itemize}
In addition, we can find a family $(\phi_a^\varepsilon)_{a \in A(\varepsilon)}$ of real-valued Lipschitz functions on $X$ such that:
\begin{itemize}
\item  The functions $\phi_a^\varepsilon$ have Lipschitz constant not greater than $\lambda/\varepsilon$ where $\lambda>0$ is a constant depending only on the doubling constant;
\item $0 \le \phi_a^\varepsilon \le 1_{B(a,2\varepsilon)}$;
\item $\sum_{a \in A(\varepsilon)}  \phi_a^\varepsilon =1$.
\end{itemize}
\end{proposition}

We note that under our assumptions $(X,d)$ is in particular separable.

\paragraph{Notation:}
\begin{enumerate}
    \item Throughout the notes, we use the letters $c,C, c_1, c_2, C_1, C_2$  to denote positive constants which may vary from line to line. 
    
    \item For two non-negative functionals $\Lambda_1, \Lambda_2$ defined on a functional space $\mathcal F$, the notation $\Lambda_1(f) \simeq \Lambda_2(f)$ means that there exist two constants $C_1, C_2>0$ such that for every $f\in \mathcal F$, $C_1\Lambda_1(f)\le \Lambda_2(f)\le C_2 \Lambda_1(f)$.
    
\item For any Borel set $A$ and any measurable function $f$, we sometimes write  the average of $f$ on the set $A$ as
\[
\fint_A f(x) d\mu(x):=\frac1{\mu(A)}\int_A f(x) d\mu(x).
\]
\item If $A,B$ are subsets of $X$, then
\[
d(A,B):=\inf \left\{ d(x,y) , x \in A, y \in B \right\}.
\]
\item If $a,b \in \mathbb R$, $a \wedge b=\min \{ a ,b \}$.
\end{enumerate}

\section{Besov-Lipschitz spaces}

We start with a short review of some properties of the Besov-Lipschitz spaces that will be useful in the sequel. The theory of Besov classes on doubling metric measure spaces is   rich and the literature on this topic is nowadays quite large so we will not try to be exhaustive and do not claim originality; For references related to the discussion below,  see for instance \cite{MR4196573},\cite{BV2},\cite{BV3},\cite{BV1}, \cite{MR3420351},\cite{MR2485404} and \cite{MR2097270}.

\subsection{Some basic properties}

For $p \ge 1, r>0, \alpha \ge 0$ and $f \in L^p(X,\mu)$, we define 
\[
E_{p,\alpha} (f,r) = \int_X \frac{1}{\mu(B(x,r))}\left( \int_{B(x,r)} \frac{|f(y)-f(x)|^p}{r^{p\alpha}} d\mu(y) \right)d\mu(x) 
\] 
and consider the Besov-Lipschitz space
\[
\mathcal{B}^{\alpha,p}(X)=\left\{ f \in L^p(X,\mu): \sup_{r >0} E_{p,\alpha} (f,r)<+\infty \right\}.
\]
We equip $\mathcal{B}^{\alpha,p}(X)$ with the norm given by
\[
\| f \|^p_{\mathcal{B}^{\alpha,p}(X)}=\| f \|^p_{L^p(X,\mu)} +  \sup_{r >0} E_{p,\alpha} (f,r).
\]

\begin{lem}\label{Lp comparison}
For every $p \ge 1$, $f \in L^p(X,\mu)$, $r >0$ and $\alpha \ge 0$
\[
E_{p,\alpha} (f,r) \le \frac{C}{r^{p\alpha}} \| f \|^p_{L^p(X,\mu)}.
\]
In particular, for $\alpha=0$, we have $\mathcal{B}^{\alpha,p}(X)=L^p(X,\mu)$.
\end{lem}

\begin{proof}
We note that
\begin{align*}
E_{p,\alpha} (f,r)& =\frac{1}{r^{p\alpha}} \int_X \frac{1}{\mu(B(x,r))}\left( \int_{B(x,r)} |f(y)-f(x)|^p d\mu(y) \right)d\mu(x) \\
 &\le \frac{2^{p-1}}{r^{p\alpha}} \int_X \frac{1}{\mu(B(x,r))}\left( \int_{B(x,r)}( |f(y)|^p+|f(x)|^p )d\mu(y) \right)d\mu(x) \\
 &= \frac{2^{p-1}}{r^{p\alpha}} \left( \| f \|^p_{L^p(X,\mu)} +\int_X \left( \int_{B(y,r)} \frac{d\mu(x)}{\mu(B(x,r))}\right) |f(y)|^p d\mu(y) \right).
\end{align*}
Using the volume doubling property, one has then
\[
\int_{B(y,r)} \frac{d\mu(x)}{\mu(B(x,r))}\le C \int_{B(y,r)} \frac{d\mu(x)}{\mu(B(x,2r))}\le C \int_{B(y,r)} \frac{d\mu(x)}{\mu(B(y,r))} =C,
\]
and the conclusion follows.
\end{proof}

Using an argument as in the proof of the previous lemma, one also has the following result:

\begin{lem}\label{local-global}
Let $p\ge1$ and  $\alpha \ge 0$. For $f \in L^p(X,\mu)$, and for every $r >0$, we have
\[
\sup_{\rho >0} E_{p,\alpha} (f,\rho) \le \frac{C}{r^{p\alpha}} \| f \|^p_{L^p(X,\mu)} +\sup_{\rho \in (0,r]}E_{p,\alpha} (f,\rho).
\]
Therefore,
\[
\mathcal{B}^{\alpha,p}(X)
=\left\{ f \in L^p(X,\mu)\, :   \limsup_{r  \to 0} E_{p,\alpha} (f,r)<+\infty \right\}.
\]
\end{lem}

It follows that for a fixed $p$, the family of spaces $\mathcal{B}^{\alpha,p}(X)$, $\alpha \ge 0$ is non-increasing:

\begin{corollary}
Let $p\ge1$ and  $\alpha \ge 0$. Then, for $\beta > \alpha$,  $\mathcal{B}^{\beta,p}(X)\subset \mathcal{B}^{\alpha,p}(X)$.
\end{corollary}

\begin{proof}
If $f \in \mathcal{B}^{\beta,p}(X)$ one has $\sup_{r >0} E_{p,\beta} (f,r)<+\infty$. This gives
\[
\limsup_{r  \to 0} E_{p, \alpha} (f,r)=\limsup_{r  \to 0} r^{p(\beta-\alpha)} E_{p,\beta} (f,r)=0<+\infty.
\]
\end{proof}

Next, we show  the Banach space property.

\begin{theorem}
 $(\mathcal{B}^{\alpha,p}(X), \| \cdot \|_{\mathcal{B}^{\alpha,p}(X)})$ is a Banach space for every $p \ge 1$ and $\alpha \ge 0$.
\end{theorem}

\begin{proof}
Let $f_n$ be a Cauchy sequence in $\mathcal{B}^{\alpha,p}(X)$. Let $f$ be the $L^p$ limit of $f_n$. From  Minkowski's inequality and Lemma \ref{Lp comparison} one has 
\begin{align*}
| E_{p,\alpha}(f,r)^{1/p}-E_{p,\alpha}(f_n,r)^{1/p}| &\le  E_{p,\alpha}(f-f_n,r)^{1/p} \le \frac{C}{r^\alpha} \| f -f_n \|_{L^p(X,\mu)}.
\end{align*}
Thus $E_{p,\alpha}(f_n,r) \to E_{p,\alpha}(f,r)$ from which we deduce
\[
E_{p,\alpha}(f,r) = \lim_{n \to +\infty}E_{p,\alpha}(f_n,r)  \le C.
\]
This implies that $f \in \mathcal{B}^{\alpha,p}(X)$ with $\| f \|_{\mathcal{B}^{\alpha,p}(X)} \le  \lim_{n \to +\infty} \| f_n \|_{\mathcal{B}^{\alpha,p}(X)}$. Similarly, for every fixed $m$, 
\[
\| f-f_m \|_{\mathcal{B}^{\alpha,p}(X)} \le  \lim_{n \to +\infty} \| f_n -f_m \|_{\mathcal{B}^{\alpha,p}(X)}
\]
and passing to the limit as $m \to +\infty$ together with the fact that  $(f_n)$ is Cauchy with respect to  $\|\cdot\|_{\mathcal{B}^{\alpha,p}(X)}$ completes the proof that $f_n \to f $ in $\mathcal{B}^{\alpha,p}(X)$ and therefore that $(\mathcal{B}^{\alpha,p}(X), \| \cdot \|_{\mathcal{B}^{\alpha,p}(X)})$ is a Banach space.
\end{proof}

\subsection{Embeddings of Besov-Lipschitz spaces into H\"older spaces}

For a fixed $p \ge 1$, one can think of the parameter $\alpha$ as a regularity parameter: The larger $\alpha$ is, the smoother functions in $\mathcal{B}^{\alpha,p}(X)$ are. The theorem below reflects this fact.
Recall that we denote by $Q$ the  constant in \eqref{eq:mass-bounds}.

\begin{theorem}\label{T:Morrey_local}
Let $\alpha >0$ and $p \ge 1$ be such that $p>\frac{Q}{\alpha}$. Let $x_0 \in X$ and $R>0$. There exists $C>0$  such that for every $f\in\mathcal{B}^{\alpha,p}(X)$,
\begin{equation}\label{E:Morrey_local}
\mu \otimes \mu -\esssup\limits_{x,y \in B(x_0,R), 0<d(x,y)<R/3}\frac{|f(x)-f(y)|}{d(x,y)^\lambda} \leq C \sup_{r \in (0,R]}E_{p,\alpha}(f,r)^{1/p}
\end{equation}
 where $\lambda=\alpha-\frac{Q}{p}$. 
\end{theorem}

\begin{proof}
Let first $0<r<R/3$ and consider $x,y\in B(x_0,R)$ with $d(x,y)\leq r$. Define
\begin{equation*}
f_r(x):=\frac{1}{\mu\big(B(x,r)\big)}\int_{B(x,r)}f(z)\,d\mu(z)
\end{equation*}
and notice that
\[
f_r(x)=\frac{1}{\mu\big(B(x,r)\big)\mu\big(B(y,r)\big)} \int_{B(x,r)}\int_{B(y,r)}f(z)\,d\mu(z')\,d\mu(z).
\]
Analogously one defines $f_r(y)$. H\"older's inequality yields
\begin{align*}
|f_r(x)-f_r(y)|&=\frac{1}{\mu\big(B(x,r)\big)\mu\big(B(y,r)\big)}\Big|\int_{B(x,r)}\int_{B(y,r)}(f(z)-f(z'))\,d\mu(z')\,d\mu(z)\Big|\\
&\leq\bigg(\frac{1}{\mu\big(B(x,r)\big)\mu\big(B(y,r)\big)}\int_{B(x,r)}\int_{B(y,r)}|f(z)-f(z')|^p\,d\mu(z')\,d\mu(z)\bigg)^{1/p}.
\end{align*}
We now note that if $z \in B(x,r)$, $z' \in B(y,r)$ then one has $d(z,z') < 3r$, $B(z,3r) \subset B(x,4r)$ and moreover from the doubling condition \eqref{eq:mass-bounds}  
\[
\mu(B(y,r )) \ge  Cr^Q \frac{ \mu(B(y,2R))}{R^Q} \ge Cr^Q \frac{ \mu(B(x_0,R))}{R^Q}.
\]
Hence, we get
\begin{align*}
|f_r(x)-f_r(y)|^p&\leq \frac{C}{r^{Q}}\int_X\frac{1}{\mu(B(z,3r))} \int_{B(z,3r)}|f(z)-f(z')|^p\,d\mu(z')\,d\mu(z)\\
&\leq Cr^{p\alpha -Q}\sup_{\rho \in(0,R/3)}\frac{1}{\rho^{p\alpha }}\int_X \frac{1}{\mu(B(z,3\rho))}\int_{B(z,3\rho)}|f(z)-f(z')|^p\,d\mu(z')\,d\mu(z)\\
&\leq Cr^{p\alpha -Q} \sup_{\rho \in (0,R]}E_{p,\alpha}(f,\rho),
\end{align*}
where the constant $C$ depends on $R$ and $\mu(B(x_0,R))$.
 Thus,
\begin{equation*}
|f_r(x)-f_r(y)| \leq Cr^{\alpha -\frac{Q}{p}} \sup_{\rho \in (0,R]}E_{p,\alpha}(f,\rho)^{1/p}.
\end{equation*}
Analogously one obtains
\begin{equation}\label{E:Morrey_local_02}
|f_{2r}(x)-f_r(x)| \leq Cr^{\alpha -\frac{Q}{p}} \sup_{\rho \in (0,R]}E_{p,\alpha}(f,\rho)^{1/p}.
\end{equation}
Let now $x\in B(x_0,R)$ be a Lebesgue point of $f$. Setting $r_k=2^{-k}r$, $k=0,1,2\ldots$, the latter inequality yields
\begin{equation}\label{E:Morrey_local_03}
|f(x)-f_r(x)|\leq\sum_{k=0}^\infty|f_{r_k}(x)-f_{r_{k+1}}(x)| \le Cr^{\alpha -\frac{Q}{p}} \sup_{\rho \in (0,R]}E_{p,\alpha}(f,\rho)^{1/p}.
\end{equation}
Let $y\in B(x_0,R)$ be another Lebesgue point of $f$ such that $d(x,y) < R/3$. Applying the triangle inequality as well as~\eqref{E:Morrey_local_02} and~\eqref{E:Morrey_local_03} with $r=d(x,y)$ we obtain
\begin{align*}
|f(x)-f(y)| &  \leq |f(x)-f_r(x)|+|f_r(x)-f_r(y)|+|f_r(y)-f(y)|\\
\leq &Cd(x,y)^{\alpha -\frac{Q}{p}} \sup_{\rho \in (0,R]}E_{p,\alpha} (f,\rho)^{1/p}.
\end{align*}
Then, by virtue of~\cite[Theorem 3.4.3]{HKST15}, the volume doubling property of the space implies the validity of the Lebesgue differentiation theorem, i.e. that $\mu$ a.e. $x \in X$ is a Lebesgue point of $f$.  The conclusion follows. 
\end{proof}

\begin{remark}
 If $X$ has maximal volume growth, i.e  $\mu( B(x,R)) \ge c R^Q$ for every $R \ge 0$, and $x \in X$,  for some $c>0$, then  after tracking the constants in the previous proof, we can let $R \to +\infty$ in \eqref{E:Morrey_local} and obtain
\[
\mu \otimes \mu -\esssup\limits_{x \neq y}\frac{|f(x)-f(y)|}{d(x,y)^\lambda} \leq C \sup_{r>0}E_{p,\alpha}(f,r)^{1/p}.
\]
The same conclusion holds if $X$ has finite diameter as can be seen by choosing $R$ large enough in Theorem \ref{T:Morrey_local}.
\end{remark}

\begin{remark}\label{regular version}
Theorem \ref{T:Morrey_local} implies that if $f \in \mathcal{B}^{\alpha,p}(X)$ with $p >\frac{Q}{\alpha}$, then one can find a  locally  $\left(\alpha-\frac{Q}{p}\right)$  H\"older continuous function $g:X \to \mathbb R$ such that $f=g$ $\mu$-a.e.
\end{remark}

\section{Korevaar-Schoen-Sobolev spaces}

For $\alpha=1$, Korevaar-Schoen-Sobolev spaces have been introduced in a Riemannian setting in \cite{MR1266480} and a presentation in a metric measure space setting is done in \cite[Section 10.4]{HKST15}. However, for some spaces like fractals, it turns out that \eqref{question intro} might be satisfied with $\alpha >1$. The parameter $\alpha$ for which it holds has to be a critical parameter in the scale of the Besov-Lipschitz spaces. In this section we study the critical exponents in the scale of the Besov-Lipschitz spaces, introduce the Korevaar-Schoen-Sobolev spaces as Besov-Lipschitz spaces at the critical parameter and prove that they satisfy  Sobolev embeddings and the whole scale of Gagliardo-Nirenberg inequalities if \eqref{question intro} is satisfied.

\subsection{Critical exponents}

\begin{defn}
Let $p \ge 1$. We define the $L^p$ \textit{critical Besov exponent} of $(X,d,\mu)$  by
\[
\alpha_p =\sup \left\{ \alpha \ge 0; \mathcal{B}^{\alpha,p}(X) \text{ contains non-constant functions} \right\}.
\]
Here and hereafter, by constant function we mean constant $\mu$-a.e.
\end{defn}

\begin{remark}
It might be that $\alpha_p=+\infty$ for every $p \ge 1$, as is the case if $(X,d)$ contains  one isolated point or is strongly  disconnected in the sense that there exist two disjoint non-empty open sets  $X_1,X_2$ such that $d (X_1, X_2)>0$, $\mu (X_1)$ is finite and  $X=X_1 \cup X_2$. Indeed, in that case the function $f=1_{X_1}$ is non-constant and  in $ \mathcal{B}^{\alpha,p}(X)$ for every $\alpha \ge 0$ and $p \ge 1$. For a sufficient condition ensuring the finiteness of $\alpha_p$, see Theorem \ref{estimate sup} below.
\end{remark}

\begin{theorem}\label{constr}
\

\begin{enumerate}
\item For every $p \ge 1$, $\alpha_p \ge 1$.
\item The map $p \to p \alpha_p$ is non-decreasing.
\item The map $p \to \alpha_p$ is non-increasing.
\end{enumerate}
In particular, if $\alpha_p$ is finite for some $p \ge 1$, then it is finite for every $p \ge 1$.
\end{theorem}

For the first item, consider $x_0 \in X$ and $y_0 \neq x_0$. Denote $r=d(x_0,y_0)$.  The function
\[
\Psi(x)=d(x, X\setminus B(x_0,r/3) )
\]
is non-constant, Lipschitz, in $L^p(X,\mu)$ for  $p \ge 1$, and  seen to be in $\mathcal{B}^{1,p}(X)$. Thus $\alpha_p \ge 1$. 

For the second item, let $p \ge 1$ and $\alpha< \alpha_p$. Let $f \in \mathcal{B}^{\alpha, p}(X)$ be non-constant  and $q \ge p$. For $n\ge 1$, denote $f_n(x)=\max \{ -n , \min   \{ f(x) , n \} \}$. Then, $f_n \in \mathcal{B}^{\alpha,p}(X)\cap L^q(X,\mu)$ and moreover
\begin{align*}
 & \int_X \frac{1}{\mu(B(x,r))}\left( \int_{B(x,r)} |f_n(y)-f_n(x)|^q d\mu(y) \right)d\mu(x) \\
 \le &  \int_X \frac{1}{\mu(B(x,r))}\left( \int_{B(x,r)} (|f_n(y)|+|f_n(x)|)^{q-p} |f_n(y)-f_n(x)|^p d\mu(y) \right)d\mu(x) \\
 \le & 2^{q-p}n^{q-p} r^{\alpha p} \int_X \frac{1}{\mu(B(x,r))}\left( \int_{B(x,r)}  \frac{ |f(y)-f(x)|^p}{r^{\alpha p}} d\mu(y) \right)d\mu(x). 
\end{align*}
Therefore $f_n \in \mathcal{B}^{\alpha \frac{p}{q}, q}(X)$. For $n$ large enough $f_n$ is not constant. Thus, $\alpha_q \ge \alpha \frac{p}{q}$. This is true for all $\alpha < \alpha_p$, so $q \alpha_q \ge p \alpha_p$.  The third item follows from the following lemma:

\begin{lem}\label{prop:convexityembed}
If $1\leq q\leq p < \infty$, there exists a constant $C>0$ such that if  $f\in \mathcal{B}^{\alpha,p}(X)$, then $|f|^{p/q}\in \mathcal{B}^{\alpha,q}(X)$ and
\begin{equation}\label{eqn:convexitybound}
 \sup_{r >0} E_{q,\alpha} (|f|^{p/q},r)  \leq C  \|f\|_{L^p(X,\mu)}^{p-q} \, \sup_{r >0} E_{p,\alpha}(f,r)^{q/p}.
\end{equation}
\end{lem}
\begin{proof} 
We use  for any $a,b \ge 0$ such that $a\neq b$, the elementary inequality
\[
\frac{|a^{p/q}-b^{p/q}|}{|a-b|}\leq \frac{p}{q} \max\{a,b\}^{\frac{p}{q}-1}.
\]
Equivalently, 
\[
|a^{p/q}-b^{p/q}|^q \leq \brak{\frac{p}{q}}^q \max\{a,b\}^{p-q} |a-b|^{q}.
\]
Using this elementary inequality, one has
 \begin{align}\label{espn:ine3}
	 & \lefteqn{\int_X \frac{1}{\mu(B(x,r))} \int_{B(x,r)} \bigl| |f(x)|^{p/q}-|f(y)|^{p/q}\bigr|^q \,d\mu(y) d\mu(x) } \quad& \notag \\
	&\leq \brak{\frac{p}{q}}^q \int_X \frac{1}{\mu(B(x,r))} \int_{B(x,r)}  \bigl( |f(x)|^{p-q}+|f(y)|^{p-q}\bigr)\bigl| |f(x)|-|f(y)|\bigr|^q \,d\mu(y) d\mu(x)   \notag \\
	&\leq \brak{\frac{p}{q}}^q \int_X \frac{1}{\mu(B(x,r))} \int_{B(x,r)}  \bigl( |f(x)|^{p-q}+|f(y)|^{p-q}\bigr)\bigl| f(x)-f(y)\bigr|^q \,d\mu(y) d\mu(x) .
\end{align}
We now observe that by Fubini's theorem and the volume doubling property
\begin{align*}
   & \int_X \frac{1}{\mu(B(x,r))} \int_{B(x,r)}  |f(y)|^{p-q} \bigl| f(x)-f(y)\bigr|^q\,\,d\mu(y) d\mu(x)  \\
   = & \int_X  \int_{B(x,r)} \frac{1}{\mu(B(y,r))}  |f(x)|^{p-q} \bigl| f(x)-f(y)\bigr|^q\,\,d\mu(y) d\mu(x) \\
   \le & C \int_X  \int_{B(x,r)} \frac{1}{\mu(B(y,2r))}  |f(x)|^{p-q} \bigl| f(x)-f(y)\bigr|^q\,\,d\mu(y) d\mu(x) \\
   \le & C  \int_X  \frac{1}{\mu(B(x,r))}  \int_{B(x,r)}  |f(x)|^{p-q} \bigl| f(x)-f(y)\bigr|^q\,\,d\mu(y) d\mu(x).
\end{align*}
Thus, applying H\"older's inequality we have
 \begin{align*}	
  & \lefteqn{\int_X \frac{1}{\mu(B(x,r))} \int_{B(x,r)} \bigl| |f(x)|^{p/q}-|f(y)|^{p/q}\bigr|^q \,d\mu(y) d\mu(x) }&\\
	\leq& C \int_X \frac{ |f(x)|^{p-q}}{\mu(B(x,r))} \left( \int_{B(x,r)}   \bigl| f(x)-f(y)\bigr|^q\,\,d\mu(y) \right) \,d\mu(x) \\
	\leq& C \int_X |f(x)|^{p-q} \Bigl( \frac{1}{\mu(B(x,r))} \int_{B(x,r)}  \bigl| f(x)-f(y)\bigr|^p d\mu(y) \Bigl)^{q/p} \, d\mu(x)\\
	\leq& C \|f\|_{L^p(X,\mu)}^{p-q} \Bigl( \int_X  \frac{1}{\mu(B(x,r))}\int_{B(x,r)}  \bigl| f(x)-f(y)\bigr|^p d\mu(y) d\mu(x) \Bigl)^{q/p},
	\end{align*}
which implies~\eqref{eqn:convexitybound}.
\end{proof}

\begin{defn}
The metric space $(X,d)$ is said to satisfy the chaining condition if there exists a constant $C_h \ge 1$ such that for every $x,y\in X$, and $n \ge 1$ there is a family of points $x_0=x,\cdots,x_n=y$ of $X$ such that for $j=0,\cdots,n-1$, $d(x_j,x_{j+1}) \le \frac{C_h}{n}d(x,y)$.
\end{defn}

For instance, geodesic spaces satisfy the chaining condition.

\begin{theorem}\label{estimate sup}
Assume that $(X,d,\mu)$ satisfies the chaining condition. Then, for every $p \ge 1$ we have  $\alpha_p \le 1+\frac{Q}{p}$.
\end{theorem}

\begin{proof}
Let $f \in \mathcal{B}^{\alpha,p}(X)$ with $\alpha >1+\frac{Q}{p}$. From Theorem \ref{T:Morrey_local} and Remark \ref{regular version}, we can assume that the function $f$ is  locally H\"older continuous with exponent $>1$ which implies that $f$ is constant in view of the chaining condition. Indeed, let $z \in X$ and $R>0$.  Let $x,y \in X$ with $d(z,x)< \frac{R}{3}$, $d(z,y)<\frac{R}{3}$ and $d(x,y)<\frac{R}{3C_h}$. From the chaining condition, for an integer $n \ge 2$ one can consider a family of points  $x_0=x,\cdots,x_n=y$ of $X$ such that for $j=0,\cdots,n-1$, $d(x_j,x_{j+1}) \le \frac{C_h}{n}d(x,y)$. One has $x_j \in B(z,R)$ and $d(x_j,x_{j+1})<\frac{R}{3}$. Therefore, from Theorem \ref{T:Morrey_local} 
\begin{align*}
|f(x)-f(y)| & \le |f(x)-f(x_1)|+\cdots+ |f(x_{n-1})-f(y)| \\
 &\le  C    \sum_{j=0}^{n-1} d(x_j,x_{j+1})^{\alpha-Q/p} \le C n \frac{1}{n^{\alpha-Q/p}} d(x,y)^{\alpha-Q/p}. 
\end{align*}
Letting $n\to +\infty$ yields $f(x)=f(y)$. By arbitrariness of $z$ and $R$ one concludes that $f$ is constant.
\end{proof}

\subsection{Korevaar-Schoen spaces and  $\mathcal{P}(p,\alpha)$ }

For $p\ge 1$, $\alpha \ge 0$, the Korevaar-Schoen space $KS^{\alpha,p}(X)$ is defined as
\[
KS^{\alpha,p}(X)=\left\{ f \in L^p(X,\mu); \limsup_{r  \to 0} E_{p,\alpha} (f,r)<+\infty \right\}
\]
equipped with the norm given by
\[
\| f \|^p_{KS^{\alpha,p}(X)}=\| f \|^p_{L^p(X,\mu)} +  \limsup_{r  \to 0} E_{p,\alpha} (f,r).
\]

From Lemma \ref{local-global}, as a set  $KS^{\alpha,p}(X)=\mathcal{B}^{\alpha,p}(X) $ and obviously $\| \cdot \|_{KS^{\alpha,p}(X)}\le \| \cdot \|_{\mathcal{B}^{\alpha,p}(X)} $.

\begin{defn}\label{property Pp}
Let $p \ge 1$, $\alpha \ge 0$. We will say that $\mathcal{P}(p,\alpha)$ holds if $KS^{\alpha,p}(X)$ contains non-constant functions and there exists a constant $C \ge 1$ such that for every $f \in KS^{\alpha,p}(X)$,
\[
\sup_{r  > 0} E_{p,\alpha} (f,r) \le C \liminf_{r  \to 0} E_{p,\alpha} (f,r).
\]
\end{defn}

\begin{remark}
We point out that the weaker property  
\[
\sup_{r  > 0} E_{p,\alpha} (f,r) \le C \limsup_{r  \to 0} E_{p,\alpha} (f,r)
\]
does not suffice to develop a rich theory, since the super-additivity of the $\liminf$ is used in crucial parts of the arguments, for instance to obtain the Sobolev embeddings of Section \ref{Section sobo}.
\end{remark}

\begin{lem}
Let $p \ge 1$, $\alpha \ge 0$. If $\mathcal{P}(p,\alpha)$ holds then $\alpha=\alpha_p$.
\end{lem}

\begin{proof}
If $\beta >\alpha$, and $f \in \mathcal{B}^{\beta,p}(X)$, we have $\liminf_{r  \to 0} E_{p,\alpha} (f,r)=0$ and thus $\sup_{r  > 0} E_{p,\alpha} (f,r)=0$ which yields that $f$ is constant.
\end{proof}

\begin{defn}
When $\alpha=\alpha_p$ the space $KS^{\alpha,p}(X)$  is referred to as  the Korevaar-Schoen-Sobolev space and  for $f \in KS^{\alpha,p}(X)$ we will denote
\begin{align}\label{p-variation}
\mathrm{Var}_p (f)= \liminf_{r  \to 0} E_{p,\alpha_p} (f,r)^{1/p}.
\end{align}
\end{defn}

An important property of $\mathrm{Var}_p$ is that it is a Sobolev quasi-seminorm in the sense of Bakry-Coulhon-Ledoux-Saloff Coste, see \cite[Section 2]{MR1386760}:

\begin{theorem}\label{L:local_norm_approx:a}
If  $\mathcal{P}(p,\alpha)$ holds, then $\mathrm{Var}_p $ is a Sobolev quasi-seminorm, i.e. it satisfies the following properties:
\begin{itemize}
\item There exists a constant $C \ge 1$ such  that for every $f,g\in KS^{\alpha,p}(X)$,
\[
\mathrm{Var}_p (f+g) \le C (\mathrm{Var}_p (f)+\mathrm{Var}_p (g));
\]
\item If $f \in KS^{\alpha,p}(X)$ is such that $\mathrm{Var}_p (f)=0$ then $f$ is constant;
\item For every $s,t \ge 0$, $\mathrm{Var}_p ((f-t)^+ \wedge s ) \le \mathrm{Var}_p (f )$;
\item There exists a constant $C>0$ such that for any nonnegative $f \in KS^{\alpha,p}(X)$  and any $\rho>1$,
\[
\left( \sum_{k \in \mathbb{Z}} \mathrm{Var}_p (f_{\rho,k})^p \right)^{1/p} \le C \mathrm{Var}_p (f),
\]
where $f_{\rho,k}:=(f-\rho^k)^+ \wedge (\rho^k(\rho-1))$, $k \in \mathbb{Z}$. 
\end{itemize}
\end{theorem}

\begin{proof}
The first two items follow from the fact that $\mathcal{P}(p,\alpha)$ implies
\[
\mathrm{Var}_p (f) \le \sup_{r  > 0} E_{p,\alpha} (f,r)^{1/p} \le C \mathrm{Var}_p (f).
\]
The third item is immediate, because $E_{p,\alpha} ((f-t)^+ \wedge s,r)  \le E_{p,\alpha} (f,r)$.  The fourth item can be proved as in the proof of Lemma 7.1 in \cite{MR1386760} and by using then the super-additivity property of $\liminf$.
\end{proof}

\begin{remark}
Under $\mathcal{P}(p,\alpha)$ the functionals $\limsup_{r \to 0} E_{p,\alpha} (f,r)^{1/p}$ and  $\sup_{r > 0} E_{p,\alpha} (f,r)^{1/p}$ are also Sobolev quasi-seminorms.
\end{remark}

\begin{remark}

If  $\mathcal{P}(p,\alpha)$ holds, then one can rewrite the conclusion of Theorem \ref{T:Morrey_local} as
\[
\mu \otimes \mu- \esssup\limits_{x,y \in B(x_0,R), 0<d(x,y)<R/3}\frac{|f(x)-f(y)|}{d(x,y)^{\alpha-\frac{Q}{p}}} \leq C \mathrm{Var}_p (f).
\]
This can be interpreted as a Morrey inequality for the Sobolev space $KS^{\alpha_p,p}(X)$.
\end{remark}

\subsection{Reflexivity and separability of the Korevaar-Schoen-Sobolev spaces}

In this section we prove that if the property $\mathcal{P}(p,\alpha_p)$ holds with $p>1$ then the  space $KS^{\alpha_p,p}(X)$ is reflexive and separable. We use the following lemma.

\begin{lem}\label{uniform}
Let $(Z,\| \cdot \|)$ be a Banach space. If for every $\varepsilon >0$ there exists $\delta >0$ with the property that $\| x+y \| \le 2 (1-\delta)$ whenever $x,y \in Z$ satisfy  $\| x \| <1$, $\| y \|<1$ and $\| x-y \|>\varepsilon$, then $(Z,\| \cdot \|)$ is reflexive. 
\end{lem}

\begin{proof}
From Milman–Pettis’ theorem, it is enough to prove that the stated property implies that $(Z,\| \cdot \|)$ is uniformly convex. Let $\varepsilon >0$ and $\delta >0$ be as in the stated property. Suppose that $x,y \in Z$ are such that $\| x \| =1$, $\| y \|=1$ and $\| x-y \|>\varepsilon$. Then for $0<\eta <1$ sufficiently close to 1 we have that $\| \eta x \| <1$, $\| \eta y \|<1$ and $\|\eta  x-\eta y \|>\varepsilon$. Therefore we obtain
\[
\| \eta x+\eta y \| \le 2 (1-\delta).
\]
Passing to the limit $\eta \to 1^{-}$ yields $\|  x+ y \| \le 2 (1-\delta)$.
\end{proof}

\begin{theorem}
Let $p>1$. If  $\mathcal{P}(p,\alpha_p)$ holds, then  $(KS^{\alpha_p,p}(X), \| \cdot \|_{KS^{\alpha_p,p}(X)})$ is a reflexive and separable Banach space.
\end{theorem}

\begin{proof}
Let $\varepsilon>0$ and suppose that $f,g\in KS^{\alpha_p,p}(X)$ satisfy $\| f \|_{KS^{\alpha_p,p}(X)} <1$, $\| g \|_{KS^{\alpha_p,p}(X)}<1$ and $\| f-g \|_{KS^{\alpha_p,p}(X)}>\varepsilon$. Since $\| f \|_{KS^{\alpha_p,p}(X)} <1$ and  $\| g \|_{KS^{\alpha_p,p}(X)}<1$, we first deduce that there exists $r_0>0$ such that for $0<r<r_0$,
\[
\|f\|_{L^p(X,\mu)}^p + E_{p,\alpha_p}(f,r) <1
\]
and 
\[
\|g\|_{L^p(X,\mu)}^p + E_{p,\alpha_p}(g,r) <1.
\]
Then, we have from the property $\mathcal{P}(p,\alpha_p)$
\begin{align*}
\varepsilon^p  <\| f-g \|^p_{KS^{\alpha_p,p}(X)} & = \|f-g\|_{L^p(X,\mu)}^p +\limsup_{r \to 0} E_{p,\alpha_p}(f-g,r)   \\
  & \le C\left( \|f-g\|_{L^p(X,\mu)}^p +\liminf_{r \to 0} E_{p,\alpha_p}(f-g,r)  \right).
\end{align*}
Therefore, there exists $r_1>0$ such that for $0<r<r_1$, 
\[
\|f-g\|_{L^p(X,\mu)}^p + E_{p,\alpha_p}(f-g,r) > \frac{\varepsilon^p}{C}.
\]
We now first assume $p \ge 2$. The Clarkson inequalities for $L^p$ functions  yield the following:
\begin{align*}
& \norm{f+g}_{L^p(X,\mu)}^p + \norm{f-g}_{L^p(X,\mu)}^p + E_{p,\alpha_p}(f+g,r) +E_{p,\alpha_p}(f-g,r)  \\
\le &2^{p-1} \left( \|f\|_{L^p(X,\mu)}^p +\|g\|_{L^p(X,\mu)}^p+ E_{p,\alpha_p}(f,r) +E_{p,\alpha_p}(g,r)\right).
\end{align*}
We therefore have for $0<r < r_0 \wedge r_1$,
\[
\norm{f+g}_{L^p(X,\mu)}^p+E_{p,\alpha_p}(f+g,r) \le 2^{p}-\frac{\varepsilon^p}{C}.
\]
This implies that 
\[
\| f+g \|^p_{KS^{\alpha_p,p}(X)} \le 2^{p}-\frac{\varepsilon^p}{C}.
\]
We then conclude from Lemma \ref{uniform} that $(KS^{\alpha_p,p}(X), \| \cdot \|_{KS^{\alpha_p,p}(X)})$ is reflexive. We now turn to the case $1<p <2$. Let $q$ be the conjugate exponent of $p$, i.e. $q=\frac{p}{p-1}$.
 We have from the reverse Minkowski inequality and Clarkson's inequalities for $L^p$ functions 
 \begin{align*}
 &\left[ \left(\norm{\frac{f+g}{2}}^p_{L^p(X,\mu)} +E_{p,\alpha_p}\left(\frac{f+g}{2},r\right)\right)^{q/p}+\left(\norm{\frac{f-g}{2}}^p_{L^p(X,\mu)} +E_{p,\alpha_p}\left(\frac{f-g}{2},r\right)\right)^{q/p}\right]^{p/q}\\
 \le & \left(\norm{\frac{f+g}{2}}^q_{L^p(X,\mu)} +\norm{\frac{f-g}{2}}^q_{L^p(X,\mu)}\right)^{p/q} + \left(E_{p,\alpha_p}\left(\frac{f+g}{2},r\right)^{q/p}+E_{p,\alpha_p}\left(\frac{f-g}{2},r\right)^{q/p}  \right)^{p/q} \\
 \le &\frac{1}{2} \|f\|_{L^p(X,\mu)}^p +\frac{1}{2}\|g\|_{L^p(X,\mu)}^p+\frac{1}{2}E_{p,\alpha_p}(f,r) +\frac{1}{2}E_{p,\alpha_p}(g,r).
 \end{align*}  
We therefore have for $0<r < r_0 \wedge r_1$,
\[
\norm{f+g}_{L^p(X,\mu)}^p+E_{p,\alpha_p}(f+g,r) \le 2^{p}\left(1-\frac{\varepsilon^q}{2^qC^{q/p}}\right)^{p/q}.
\]
This implies that 
\[
\| f+g \|^p_{KS^{\alpha_p,p}(X)} \le 2^{p}\left(1-\frac{\varepsilon^q}{2^qC^{q/p}}\right)^{p/q},
\]
and we conclude as above. It remains to prove separability. The identity map $\iota: (KS^{\alpha_p,p}(X), \| \cdot \|_{KS^{\alpha_p,p}(X)}) \to (L^{p}(X,\mu),  \| \cdot \|_{L^{p}(X,\mu)})$ is a linear and bounded injective map. Since the space $(KS^{\alpha_p,p}(X), \| \cdot \|_{KS^{\alpha_p,p}(X)}) $ is reflexive and $L^{p}(X,\mu)$ is separable because $X$ is, it now follows from Proposition 4.1 in \cite{alvarado2023simple} that  $(KS^{\alpha_p,p}(X), \| \cdot \|_{KS^{\alpha_p,p}(X)}) $ is separable.
\end{proof}

\subsection{Sobolev type embeddings / Gagliardo-Nirenberg inequalities}\label{Section sobo}

Let $p \ge 1$, $\alpha >0$. Throughout the section we assume:
\begin{itemize}
\item A volume non-collapsing condition: There exists $R>0$ such that
\[
\inf_{x \in X} \mu(B(x,R)) >0.
\] 
\item The property $\mathcal{P}(p,\alpha)$ holds (see Definition \ref{property Pp}).
\end{itemize}

Note that the non-collapsing condition always holds if $X$ has a finite diameter since we can take $R=\mathrm{diam} X$. Also note that  the collapsing condition implies from \eqref{eq:mass-bounds} that for every $x \in X$, $0 < r \le R$,
\[
\mu(B(x,r)) \ge c r^Q,
\]
with $c>0$ (depending on $R$). For $f \in L^q(X,\mu)$, $q \ge 1$ and $r>0$  we consider  the averaging operator
\[
\mathcal{M}_r f (x) =\frac{1}{\mu(B(x,r))} \int_{B(x,r)} f(y) d\mu(y).
\]

\begin{lem}
There exists a constant $C>0$ such that for every $r \in (0,R]$, $q \ge 1$ and $f \in L^{q}(X,\mu)$,
\[
\| \mathcal{M}_r f \|_{L^\infty (X,\mu)} \le \frac{C}{r^{Q/q}} \, \|  f \|_{L^q (X,\mu)}.
\]
\end{lem}

\begin{proof}
The estimate follows from H\"older's inequality and the non-collapsing condition.
\end{proof}

\begin{lem}
There exists a constant $C>0$ such that for every $r >0$ and $f \in KS^{\alpha,p}(X)$,
\[
\| f -\mathcal{M}_r f \|_{L^p(X,\mu)} \le C r^\alpha \, \mathrm{Var}_p (f).
\]
\end{lem}
\begin{proof}
This follows from the property $\mathcal{P}(p,\alpha)$. Indeed, from H\"older's inequality
\[
\| f -\mathcal{M}_r f \|_{L^p(X,\mu)} \le  r^\alpha \, \sup_{\rho>0} E_{p,\alpha_p} (f,\rho)^{1/p} \le C r^\alpha \, \mathrm{Var}_p (f).
\]
\end{proof}
Remarkably, together with Theorem \ref{L:local_norm_approx:a}, the two simple previous lemmas are enough to obtain the full scale of Gagliardo-Nirenberg inequalities. The results follow from applying the results of \cite[Theorem 9.1]{MR1386760}, see also \cite{MR4196573}. 

\begin{theorem}\label{sobo}
Let $q=\frac{pQ}{Q-\alpha p}$ with the convention that $q=\infty$ if $Q=\alpha p$. Let $ r,s \in (0,+\infty]$ and $\theta \in (0,1]$ satisfying 
\[
\frac{1}{r}=\frac{\theta}{q}+\frac{1-\theta}{s}.
\]
If $Q=\alpha p$ with $p>1$, we assume $r<+\infty$. Then, there  exists a constant $C>0$ such that for every $f \in KS^{\alpha,p}(X)$,
\begin{equation}\label{E:Gagliardo-Nirenberg:corollary1_loc}
  \| f \|_{L^r(X,\mu)} \le C \left(\|f\|_{L^p(X,\mu)}+ \mathrm{Var}_p (f)\right)^{\theta}\|f\|_{L^s(X,\mu)}^{1-\theta}.
\end{equation}
\end{theorem}

We explicitly point out  some particular cases of interest.
\begin{enumerate}
\item Assume that  $p \alpha  <Q$. If $r=s$, then $r=\frac{pQ}{Q-p \alpha}$ and~\eqref{E:Gagliardo-Nirenberg:corollary1_loc} yields the  Sobolev inequality
\[
\| f \|_{L^r(X,\mu)} \le C \left( \|f\|_{L^p(X,\mu)}+ \mathrm{Var}_p (f)\right).
\]
\item Assume that  $p \alpha  <Q$. If $s=+\infty$ and $r \ge \frac{pQ}{Q-p\alpha }$, then~\eqref{E:Gagliardo-Nirenberg:corollary1_loc} yields 
\[
\| f \|_{L^r(X,\mu)} \le C \left(\|f\|_{L^p(X,\mu)}+ \mathrm{Var}_p (f)\right)^{\theta} \| f \|^{1-\theta}_{L^\infty(X,\mu)}
\]
with $\theta=\frac{pQ}{r(Q-p \alpha)}$.
\item If $r=p>1$ and $s=1$, then~\eqref{E:Gagliardo-Nirenberg:corollary1_loc} yields the  Nash inequality
\[
\| f \|_{L^p(X,\mu)} \le C \left(\|f\|_{L^p(X,\mu)}+ \mathrm{Var}_p (f)\right)^{\theta} \| f \|^{1-\theta}_{L^1(X,\mu)}
\]
with $\theta=\frac{(p-1)Q}{p(\alpha+Q)-Q}$.
\item Assume  either  $p \alpha  > Q$ or $p \alpha  = Q$ with $p=1$. Then, for   $s \ge 1$,
\begin{equation*}
\| f \|_{L^\infty(X,\mu)} \le C \left(\|f\|_{L^p(X,\mu)}+\mathrm{Var}_p (f)\right)^{\theta} \| f \|^{1-\theta}_{L^s(X,\mu)},
\end{equation*}
where $\theta =\frac{pQ}{pQ+s(p\alpha-Q)}$. In particular, if $s=1$, and if $f$ is supported in a set $\Omega$ of finite measure we have  $\| f \|_{L^s(X,\mu)}\le \| f \|_{L^\infty(X,\mu)} \mu \left(\Omega \right)$ and we get:
\[
\| f \|_{L^\infty(X,\mu)} \le  C \left(\|f\|_{L^p(X,\mu)}+\mathrm{Var}_p (f)\right) \mu(\Omega)^{\frac{\alpha}{Q}-\frac{1}{p}}.
\]

\end{enumerate}

From \cite[Corollaries 6.3 \& 6.4]{MR1386760},  in the case $p \alpha =Q$ with $p>1$  one also obtains Trudinger-Moser type inequalities. 

\begin{corollary}\label{C:Trudinger-Moser_1_loc}

Assume that $p\alpha =Q$ and that $p>1$. Let $k \ge p-1$ be an integer. Then, there exist constants $c,C>0$ such that for every $f \in KS^{\alpha,p}(X)$ with $\|f\|_{L^p(X,\mu)}+ \mathrm{Var}_p (f) = 1$,
\begin{align*}
\int_X \exp_k \left( c |f|^{\frac{p}{p-1}} \right) d\mu  \le C \| f \|^p_{L^p(X,\mu)},
\end{align*}
where $\exp_k (x)=\sum_{\ell=k}^{+\infty} \frac{x^\ell}{\ell!}$. Moreover, if $f \in KS^{\alpha,p}(X)$ with $\|f\|_{L^p(X,\mu)}+ \mathrm{Var}_p (f) = 1$ is supported in a set $\Omega$ of finite measure then
\[
\int_\Omega e^{ c |f|^{\frac{p}{p-1}} } d\mu \le C \mu(\Omega).
\]
\end{corollary}

\begin{remark}
In all of those inequalities, it is  possible to track the dependence of the constants on $R,Q,\inf_{x \in X} \mu(B(x,R))$ and the constant in the property $\mathcal{P}(p,\alpha)$, see the arguments in \cite{MR4196573}.
\end{remark}
\begin{remark}
 If $X$ has maximal volume growth, i.e. $\mu( B(x,R)) \ge c R^Q$ for every $R>0$ and $x\in X$ for some $c>0$ then we can let $R \to +\infty$ in the arguments yielding the Gagliardo-Nirenberg inequalities and get everywhere $\mathrm{Var}_p (f)$ instead of $\|f\|_{L^p(X,\mu)}+\mathrm{Var}_p (f)$, see the arguments in \cite{MR4196573} which follow again from \cite[Theorem 9.1]{MR1386760}.
\end{remark}

\section{Korevaar-Schoen-Sobolev spaces and Poincar\'e inequalities}

\subsection{Poincar\'e inequalities and $\mathcal{P}(p,1)$}

As before, $(X,d,\mu)$ is a metric measure space where $\mu$ is a positive and doubling Borel regular measure.
In this section, under the assumption of a $p$-Poincar\'e inequality, we prove the property $\mathcal{P}(p,\alpha)$ with $\alpha=1$. Let $p \ge 1$. Consider the following $p$-Poincar\'e inequality for  locally Lipschitz functions
\begin{align}\label{p-Poincare}
\int_{B(x,r)} | f(y) -f_{B(x,r)}|^p d\mu (y) \le C r^p \int_{B(x,\lambda r)} (\mathrm{Lip} f )(y)^p d\mu (y)
\end{align}
where we denote
\[
(\mathrm{Lip} f )(y)=\limsup_{r \to 0 } \sup_{x \in X, d(x,y) \le r} \frac{|f(x)-f(y)|}{r}
\]
and
\[
f_{B(x,r)}=\frac{1}{\mu(B(x,r))} \int_{B(x,r)}  f(y) d\mu(y).
\]
In the inequality, the constants $C>0$ and $\lambda \ge 1$ are independent from $x$, $r$ and $f$.

\begin{remark}
Poincar\'e inequalities and their applications in the study of metric spaces have  extensively been studied in the literature and are nowadays standard assumptions, see for instance \cite{HKST15} and \cite{MR1800917} for detailed accounts. For concrete examples, it is known for instance that if a metric measure space satisfies a measure contraction property $\mathrm{MCP}(0,N)$ for some $N \ge 1$, then the $p$-Poincar\'e inequality holds for every $p \ge 1$, see \cite{MR1617040} and  \cite{MR2237206,MR2237207}. As a consequence, complete Riemannian manifolds with non-negative Ricci curvature and many sub-Riemannian manifolds support a $p$-Poincar\'e inequality.
\end{remark}

\begin{remark}
 In view of the Haj\l{}asz-Koskela Sobolev embedding \cite[Theorem 5.1]{MR1683160} (see also Theorem 9.1.2 in \cite{HKST15}), for $p>1$, one can replace the assumption of a $p$-Poincar\'e inequality \eqref{p-Poincare} by the assumption of a $(1,p)$ Poincar\'e inequality:
 \begin{align}\label{weak Poincare}
\fint_{B(x,r)} | f(y) -f_{B(x,r)}| d\mu (y) \le C r \left( \fint_{B(x,\lambda r)} (\mathrm{Lip} f )(y)^p d\mu (y)\right)^{1/p}.
\end{align}
\end{remark}

\begin{remark}\label{R:pPI_Lip_vs_ug}
If $(X,d)$ is complete, the $p$-Poincar\'e inequality~\eqref{p-Poincare} is known to be equivalent to the $p$-Poincar\'e inequality with upper gradients, c.f.~\cite[Theorem 8.4.2]{HKST15}.
\end{remark}

The  main result in that setting is the following theorem. It follows from a combination of results in \cite{MR2415381} and \cite{MR1628655} (see also \cite{GKS}).  We define $L^p_{loc}(X,\mu)$ to be the space of locally $p$-integrable functions and for $f_n,f \in L^p_{loc}(X,\mu)$ we say that $f_n \to f$ in $L^p_{loc}(X,\mu)$ if for every ball $B \subset X$ one has $\int_B | f_n-f|^p d\mu \to 0$ when $n \to +\infty$.

\begin{theorem}\label{equi poinc}
The $p$-Poincar\'e inequality \eqref{p-Poincare} implies $\mathcal{P}(p,1)$. Moreover, on $KS^{1,p}(X)$
\[
 \mathrm{Var}_p (f)^p \simeq \inf_{f_n} \liminf_{n \to +\infty} \int_{X} (\mathrm{Lip} f_n )(y)^p d\mu (y)
\]
where the infimum is taken over the sequences of locally Lipschitz functions $f_n$ such that $f_n \to f$ in $L^p_{loc}(X,\mu)$. 
\end{theorem}

\begin{proof}
The proof is a minor modification of the proof of Theorem 3.1 in \cite{MR3558354}; We however write all details since similar arguments will be used in the next section in a more complicated setting.
Fix $r>0$ and, see Proposition \ref{partition of unity}, consider an  $r$-covering of $X$ that consists of balls $\{B(x_i,r)\}_{i\geq 1}$ with the property that  $\{B(x_i,2\lambda r)\}_{i\geq 1}$ have the bounded overlap property, i.e. there exists $C>0$ (independent from $r$) such that
\[
\sum_{i\geq 1}\mathbf{1}_{B(x_i,2\lambda r)}(x)<C
\]
for all $x\in X$. In addition, for any $x\in B(x_i,r)$ and $y\in B(x,r)$ we note that the doubling property implies
\begin{align*}
&\mu(B(x_i,r))\leq \mu(B(y,4 r))\leq C \mu(B(y,r)),\\
&\mu(B(x_i,r))\leq\mu(B(x,2r))\leq C \mu(B(x,r)).    
\end{align*}
Now, let $f$ be a locally Lipschitz function on $X$ which is in $L_{loc}^p(X,\mu)$. We have
\begin{align*}
  & \frac{1}{r^p}\int_X\int_{B(x,r)}\frac{|f(x)-f(y)|^p}{\mu(B(x,r))}d\mu(y)\,d\mu(x)  \\
  \leq & \frac{1}{r^p}\sum_{i\geq 1}\int_{B(x_i,r)}\int_{B(x,r)}\frac{|f(x)-f(y)|^p}{\mu(B(x,r))}d\mu(y)\,d\mu(x)\\
 \leq & \frac{2^{p-1}}{r^p}\sum_{i\geq 1}\int_{B(x_i,r)}\int_{B(x,r)}\frac{|f(x)-f_{B(x_i,r)}|^p}{\mu(B(x,r))} + \frac{|f(y)-f_{B(x_i,r)}|^p}{\mu(B(x,r))}d\mu(y)\,d\mu(x).
\end{align*}
We control the first term with the $p$-Poincar\'e inequality as follows.
\begin{align*}
 &\sum_{i\geq 1}\int_{B(x_i,r)}\int_{B(x,r)}\frac{|f(x)-f_{B(x_i,r)}|^p}{\mu(B(x,r))} d\mu(y) \,d\mu(x) \\
 =&\sum_{i\geq 1}\int_{B(x_i,r)} |f(x)-f_{B(x_i,r)}|^p d\mu(x) \\
 \le &C r^p \sum_{i\geq 1} \int_{B(x_i,\lambda r)} (\mathrm{Lip} f )(y)^p d\mu (y) \le C r^p \int_{X} (\mathrm{Lip} f )(y)^p d\mu (y).
\end{align*}
The second term can be controlled in a similar way. First, by using Fubini's theorem and the volume doubling property one obtains
\begin{align*}
  \sum_{i\geq 1}\int_{B(x_i,r)}\int_{B(x,r)} \frac{|f(y)-f_{B(x_i,r)}|^p}{\mu(B(x,r))}d\mu(y)\,d\mu(x) 
 \le & \sum_{i\geq 1}\int_{B(x_i,2r)}\int_{B(y,r)} \frac{|f(y)-f_{B(x_i,r)}|^p}{\mu(B(x,r))}d\mu(x)\,d\mu(y) \\
 & \le C \sum_{i\geq 1}\int_{B(x_i,2r)} |f(y)-f_{B(x_i,r)}|^p d\mu(y). 
\end{align*}
Then, one has
\begin{align*}
 & \int_{B(x_i,2r)} |f(y)-f_{B(x_i,r)}|^p d\mu(y) \\
 \le & 2^{p-1} \left(  \int_{B(x_i,2r)} |f(y)-f_{B(x_i,2r)}|^p d\mu(y)+ \mu( B(x_i,2r))  |f_{B(x_i,2r)}-f_{B(x_i,r)}|^p \right) \\
 \le & C  \left(  r^p \int_{B(x_i,2\lambda r)} \mathrm{Lip}(f) (y)^p d\mu(y)+ \mu( B(x_i,2r))  |f_{B(x_i,2r)}-f_{B(x_i,r)}|^p \right).
\end{align*}
Finally, from H\"older's inequality and the $p$-Poincar\'e inequality again we get
\begin{align*}
  \mu( B(x_i,2r))  |f_{B(x_i,2r)}-f_{B(x_i,r)}|^p &  \le  C \int_{B(x_i,r)} | f(y) -f_{B(x_i,2r)}|^p d\mu(y) \\
  & \le  C \int_{B(x_i,2r)} | f(y) -f_{B(x_i,2r)}|^p d\mu(y) \\
&  \le C r^p  \int_{B(x_i,2\lambda r)} (\mathrm{Lip} f )(y)^p d\mu (y).
\end{align*}
Combining everything together we obtain that for every $r>0$
\begin{align}\label{sup kor so}
\frac{1}{r^p}\int_X\int_{B(x,r)}\frac{|f(x)-f(y)|^p}{\mu(B(x,r))}d\mu(y)\,d\mu(x) \le C \int_{X} (\mathrm{Lip} f )(y)^p d\mu (y).
\end{align}
We therefore proved that any locally Lipschitz function which is in $L^p(X,\mu)$ and such that $\mathrm{Lip} f \in L^p(X,\mu)$  is in the Besov-Lipschitz space $\mathcal{B}^{1,p}(X)$. In particular, $\mathcal{B}^{1,p}(X)$ contains non-constant functions. The estimate \eqref{sup kor so} also shows that for every $f  \in L^p(X,\mu)$ and every ball $B$
\[
\frac{1}{r^p}\int_B \int_{B(x,r)}\frac{|f(x)-f(y)|^p}{\mu(B(x,r))}d\mu(y)\,d\mu(x) \le C \inf_{f_n} \liminf_{n \to +\infty} \int_{X} (\mathrm{Lip} f_n )(y)^p d\mu (y)
\]
where the infimum is taken over the sequences of locally Lipschitz functions $f_n$ such that $f_n \to f$ in $L^p_{loc}(X,\mu)$. Indeed we have
\[
\lim_{n \to +\infty} \int_B \int_{B(x,r)}\frac{|f_n(x)-f_n(y)|^p}{\mu(B(x,r))}d\mu(y)\,d\mu(x)=\int_B \int_{B(x,r)}\frac{|f(x)-f(y)|^p}{\mu(B(x,r))}d\mu(y)\,d\mu(x).
\]
 This proves that for every $f  \in L^p(X,\mu)$
\[
\frac{1}{r^p}\int_X \int_{B(x,r)}\frac{|f(x)-f(y)|^p}{\mu(B(x,r))}d\mu(y)\,d\mu(x) \le C \inf_{f_n} \liminf_{n \to +\infty} \int_{X} (\mathrm{Lip} f_n )(y)^p d\mu (y).
\]
We now turn to the second part of the proof where we establish that
\[
\inf_{f_n} \liminf_{n \to +\infty} \int_{X} (\mathrm{Lip} f_n )(y)^p d\mu (y) \le C\liminf_{r \to 0} \frac{1}{r^p}\int_X\int_{B(x,r)}\frac{|f(x)-f(y)|^p}{\mu(B(x,r))}d\mu(y)\,d\mu(x). 
\]

Fix $\eps>0$. Let $\{B_i^\eps=B(x_i,\eps)\}_i$ be an $\eps$-covering of $X$, so that   the family $\{B_i^{5\eps}\}_i$ has the bounded overlap property uniformly in $\eps$. Let $\pip_i^\eps$ be a $(C/\eps)$-Lipschitz partition of unity subordinated to this cover, see Proposition \ref{partition of unity}: that is, 
$0\le \pip_i^\eps\le 1$ on $X$, $\sum_i\pip_i^\eps=1$ on $X$, and $\pip_i^\eps=0$ in $X\setminus B_i^{2\eps}$. For $f \in KS^{1,p}(X)$, we  set 
\[
f_\eps:=\sum_i f_{B_i^\eps}\, \pip_i^\eps,
\]
where $f_{B_i^\eps}=\vint_{B_i^\eps} fd\mu$. Then $f_\eps$ is locally Lipschitz. Indeed, for $x,y\in B_j^\eps$ we see that
\begin{align*}
|f_\eps(x)-f_\eps(y)|&=\left| \sum_{i:2B_i^\eps\cap 2B_j^\eps\ne\emptyset} (f_{B_i^\eps}-f_{B_j^\eps})(\pip_i^\eps(x)-\pip_i^\eps(y)) \right|\\
&\le \sum_{i:2B_i^\eps\cap 2B_j^\eps\ne\emptyset}|f_{B_i^\eps}-f_{B_j^\eps}||\pip_i^\eps(x)-\pip_i^\eps(y)|\\
 &\le \frac{C\, d(x,y)}{\eps} \sum_{i:2B_i^\eps\cap 2B_j^\eps\ne\emptyset}
    \left(\vint_{B_i^\eps}\vint_{B(w,6\eps)}|f(u)-f(w)|^p\, d\mu(u)\, d\mu(w)\right)^{1/p}.
\end{align*}
Therefore, we see that on $B_j^\eps$
\begin{align*}
\mathrm{Lip} (f_\eps)&\le \frac{C}{\eps}\sum_{i:2B_i^\eps\cap 2B_j^\eps\ne\emptyset}
    \left(\vint_{B_i^\eps}\vint_{B(x,6\eps)}|f(y)-f(x)|^p\, d\mu(y)\, d\mu(x)\right)^{1/p}\\
    &\le C\left(\vint_{5B_j^\eps} \vint_{B(x,6\eps)}\frac{|f(y)-f(x)|^p}{\eps^p}\, d\mu(y)\, d\mu(x)\right)^{1/p},
\end{align*}
and so by the bounded overlap property of the collection $5B_j^\eps$,
\begin{align*}
\int_X\mathrm{Lip} (f_\eps)^p\, d\mu &\le \sum_j \int_{B_j^\eps}\mathrm{Lip} (f_\eps)^p\, d\mu\\
  &\le C\, \sum_j \int_{5B_j^\eps} \vint_{B(x,6\eps)}\frac{|f(y)-f(x)|^p}{\eps^p}\, d\mu(y)\, d\mu(x)\\
  &\le C\, \int_X \vint_{B(x,6\eps)}\frac{|f(y)-f(x)|^p}{\eps^p}\, d\mu(y)\, d\mu(x).
 \end{align*}
Hence we have 
\begin{equation}\label{eq:sup-gradient}
\liminf_{\eps \to 0} \int_X\mathrm{Lip} (f_\eps)^p\, d\mu  \le C\liminf_{\eps \to 0} \frac{1}{\eps^p}\int_X\int_{B(x,\eps)}\frac{|f(x)-f(y)|^p}{\mu(B(x,\eps))}d\mu(y)\,d\mu(x) <+\infty.
\end{equation}
In a similar manner, we can also show that 
\[
\int_X|f_\eps(x)-f(x)|^p\, d\mu(x)\le C \eps^p \int_X\int_{B(x,6\eps)}\frac{|f(x)-f(y)|^p}{\eps^p \mu(B(x,\eps))}d\mu(y)\,d\mu(x) .
\]
Therefore $f_{\eps} \to f$ in $L^p(X,\mu)$. By now the proof is  complete.
\end{proof}

\begin{remark}\label{Newton KS}
It follows from Theorem \ref{equi poinc} and \cite{MR1708448} (or \cite[Theorem 10.1.1]{HKST15}) that if $p>1$ and the $p$-Poincar\'e inequality is satisfied,  then the Korevaar-Schoen-Sobolev space $KS^{1,p}(X)$ coincides (with equivalent norm) with the Newtonian Sobolev space $N^{1,p}(X)$ introduced by Shanmugalingam in \cite{Shanmu00}. On the other hand, if $p=1$ and the $1$-Poincar\'e inequality is satisfied,  then the Korevaar-Schoen-Sobolev space $KS^{1,1}(X)$ coincides (with equivalent norm) with the BV space introduced by Miranda in \cite{MR2005202}; This fact was first observed in \cite{MR3558354}. It follows that if $p \ge 1$ and the $p$-Poincar\'e inequality is satisfied then $KS^{1,p}(X)$ is dense in $L^p(X,\mu)$.
\end{remark}

\begin{remark}\label{existence limit poincare}
In the previous theorem, the property $\mathcal{P}(p,1)$ implies in particular that
\[
\limsup_{\eps \to 0} \frac{1}{\eps^p}\int_X\int_{B(x,\eps)}\frac{|f(x)-f(y)|^p}{\mu(B(x,\eps))}d\mu(y)\,d\mu(x) \le C \liminf_{\eps \to 0} \frac{1}{\eps^p}\int_X\int_{B(x,\eps)}\frac{|f(x)-f(y)|^p}{\mu(B(x,\eps))}d\mu(y)\,d\mu(x).
\]
It is therefore natural to ask whether or not the limit actually exists, i.e. if the inequality holds with $C=1$. It has been recently proved in \cite{MR4375837} (see also \cite{han2021asymptotic}) that under the additional condition that the tangent space in the Gromov-Hausdorff sense is Euclidean with fixed dimension, the limit  exists if $p>1$ and $f \in KS^{1,p}(X)$ and is given by
\[
\lim_{\eps \to 0} \frac{1}{\eps^p}\int_X\int_{B(x,\eps)}\frac{|f(x)-f(y)|^p}{\mu(B(x,\eps))}d\mu(y)\,d\mu(x) = \mathrm{Ch}_p (f)
\]
where $\mathrm{Ch}_p$ is (a multiple of) the Cheeger $p$-energy.
\end{remark}

\begin{remark}
In the previous proof, the upper bound
\[
   \inf_{f_n} \liminf_{n \to +\infty} \int_{X} (\mathrm{Lip} f_n )(y)^p d\mu (y) \le C \liminf_{\eps \to 0} \frac{1}{\eps^p}\int_X\int_{B(x,\eps)}\frac{|f(x)-f(y)|^p}{\mu(B(x,\eps))}d\mu(y)\,d\mu(x)
\]
does not use the $p$-Poincar\'e inequality and therefore always holds in volume doubling metric measure  spaces.
\end{remark}

\begin{remark}
Using heat kernel techniques, the following has been proved in \cite{BV2}:
\begin{itemize}
\item If the 2-Poincar\'e inequality and a weak Bakry-\'Emery estimate are satisfied then $\mathcal{P}(1,1)$ holds;
\item If the 2-Poincar\'e inequality and a quasi Bakry-\'Emery estimate are satisfied then $\mathcal{P}(p,1)$ holds for every $p \ge 1$.
\end{itemize}
\end{remark}

\begin{remark}
If $(X,d)$ is complete and $p>1$, then the upper bound  \eqref{sup kor so} also follows from arguments on maximal functions. Indeed, from the Keith-Zhong theorem \cite[Theorem 12.3.9]{HKST15},  the $p$-Poincar\'e inequality \eqref{p-Poincare} implies a $(1,q)$-Poincar\'e inequality for some $1 \le q <p$. From \cite[Theorem 8.1.7]{HKST15}, this $q$-Poincar\'e inequality implies the pointwise estimate
\begin{align}\label{pointwise}
|f(x)-f(y)|\le C d(x,y) ( \mathcal{M} ((\mathrm{Lip} f)^q)(x)+\mathcal{M} ((\mathrm{Lip} f)^q)(y))^{1/q}
\end{align}
where 
\[
 \mathcal{M} ((\mathrm{Lip} f)^q)(x)=\sup_{r >0} \fint_{B(x,r)} (\mathrm{Lip} f)^q(y) d\mu(y)
\]
is the maximal function associated to $(\mathrm{Lip} f)^q$. Since $q<p$, from $L^{p/q}$-boundedness of the maximal function one has 
\begin{align}\label{maximal}
\int_X \mathcal{M}( (\mathrm{Lip} f)^q)^{p/q} (y) d\mu(y) \le C \int_X (\mathrm{Lip} f)^p (y) d\mu(y)
\end{align}
and \eqref{sup kor so} then directly follows from \eqref{pointwise} and \eqref{maximal}.
\end{remark}

\subsection{Generalized Poincar\'e inequalities and controlled cutoffs}

In this section, we are interested in sufficient conditions for $\mathcal{P}(p,\alpha)$, where the parameter $\alpha$ is possibly greater than one. We make a further assumption on the space $(X,d)$ and assume that it is compact. Concerning the measure $\mu$ we  assume that it is a Radon measure and still assume that it is doubling. We denote by $C(X)$ the space of continuous functions on $X$ and by $\mathcal{B}(X)$ the class of Borel sets in $X$.
\begin{defn}
A local transition kernel $\{\rho_n , n \in \mathbb{N} \}$ on $X$  is a sequence of Radon measures
\[
\rho_n : \mathcal{B} (X ) \otimes \mathcal{B} (X ) \rightarrow \mathbb{R}_{\ge 0}
\]
such that for any closed sets $A,B \subset X$ with $d(A,B)>0$
\[
\lim_{n \to +\infty}  \int_A \int_{B} d\rho_n(x, y) =0.
\]
\end{defn}

\begin{example}
\
\begin{itemize}
\item Define for $r>0$, and $\alpha \ge 0$ the Korevaar-Schoen transition kernel
\[
d\rho_r(x,y)= \frac{1_{B(x,r)} (y)}{r^\alpha \mu (B(x,r))} d\mu(y) d\mu(x).
\]
Then, the locality property
\[
\lim_{r \to 0}  \int_A \int_B d\rho_r(x,y)=0
\]
 when $d(A,B)>0$ is easily checked.
\item Consider the graph approximation $V_r$, $r>0$ of $(X,d,\mu)$ as in Section 3 of \cite{MR2161694}. For $\alpha \ge 0$, consider the transition kernel defined by
\[
\rho_r (A,B)=\frac{1}{r^\alpha} \mathrm{Card} \left\{  (x,y) \in (A \cap V_r) \times (B \cap V_r):   x \sim y\right\}, \quad r>0,
\]
where $x \sim y$ means that $x$ and $y$ are neighbors in the graph $V_r$, and $\mathrm{Card}$ denotes the number of elements in the set. Then, the locality property
\[
\lim_{r \to 0}  \int_A \int_B d\rho_r(x,y)=0
\]
 when $d(A,B)>0$ is also easily checked.
\end{itemize}
\end{example}

Given a local transition kernel $\rho_n$ and $p  \ge1$, for $f \in C(X)$ we consider the sequence of Radon measures
\[
\nu_{n,p}(f,A)= \int_A \int_X | f(x)-f(y)|^p d\rho_n(x, y) , \quad A \in \mathcal{B}(X).
\]
We will denote
\[
\mathcal{F}_p=\left\{ f \in C(X):\, \sup_{n} \nu_{n,p}(f,X) <+\infty \right\}.
\]

 Let $\alpha >0$. To prove the property $\mathcal{P}(p,\alpha)$ in that setting, we  consider the following two conditions:

\begin{defn}\label{Poincloc}
We will say that the generalized $p$-Poincar\'e inequality holds if there exists $C>0$ and $\lambda \ge 1$ such that for every $f\in \mathcal{F}_{p}$, $x \in X$, $r>0$,
\begin{align}\label{p-Poincare generalized}
\int_{B(x,r)} | f(y) -f_{B(x,r)}|^p d\mu (y) \le C r^{p\alpha} \liminf_{n \to +\infty} \nu_{n,p} (f,B(x,\lambda r))
\end{align}
\end{defn}
\begin{defn}\label{CCloc} 
 We will say that the controlled cutoff condition holds if for every $\eps >0$ there exists a  covering $\{B_i^\eps=B(x_i,\eps)\}_i$  of $X$, so that   the family $\{B_i^{5\eps}\}_i$ has the bounded overlap property (uniformly in $\eps$) and  an associated family of functions $\pip_i^\eps$ such that:
  \begin{itemize}
 \item $\pip_i^\eps \in \mathcal{F}_{p}$;
\item $0\le \pip_i^\eps\le 1$ on $X$;
\item $\sum_i\pip_i^\eps=1$ on $X$; 
\item $\pip_i^\eps=0$ in $X\setminus B_i^{2\eps}$;
\item $ \limsup_{n \to +\infty} \nu_{n,p} (\pip_i^\eps,X)  \le C \frac{\mu(B_i^{\eps})}{\eps^{\alpha p}}$.
\end{itemize}
\end{defn}

\begin{remark}
Even though  Definition \ref{CCloc}  might seem difficult to check at first, it is in essence a capacity estimate requirement for balls. For instance, assume that for every ball $B$ with radius $\varepsilon$ one can find a non-negative $\phi \in  \mathcal F_p$  supported inside of $B$  with  $\phi=1$ on $ B/2$ such that
\[
 \limsup_{n \to +\infty} \nu_{n,p} (\phi,X)\le C \frac{\mu(B)}{\eps^{\alpha p}}.
\]
Then the controlled cutoff condition of Definition \ref{CCloc}  is proved to be satisfied using covering by balls satisfying the bounded overlap property as in Section 4.1 in \cite{HKST15}; see also Lemma 2.5 in \cite{MR4099475} for a related discussion.
\end{remark}

\begin{remark}
Definition \ref{Poincloc} is a generalized Poincar\'e on  balls and Definition \ref{CCloc} involves a covering of the space by balls. However, the examples of nested fractals in Section \ref{section fractal} show that in some situations it is more convenient to work with other basis of the topology, like simplices in the case of fractals.
\end{remark}

We now show that the combination of the previous conditions implies the property $\mathcal{P}(p,\alpha)$.

\begin{theorem}\label{equi poinc gene}
The  generalized $p$-Poincar\'e inequality \eqref{p-Poincare generalized} and  the controlled cutoff condition imply $\mathcal{P}(p,\alpha)$. Moreover, in that case, on $KS^{\alpha,p}(X)$
\[
 \mathrm{Var}_p (f)^p \simeq \inf_{f_m} \liminf_{m \to +\infty} \liminf_{n \to +\infty} \nu_{n,p}(f_m,X)
\]
where the infimum is taken over the sequences of functions $f_m \in \mathcal{F}_{p}$ such that $f_m \to f$ in $L^p (X,\mu)$. 
\end{theorem}

\begin{proof}
Repeating the arguments of the first part of the proof of Theorem \ref{equi poinc} shows that the generalized $p$-Poincar\'e inequality implies that  for every $f \in \mathcal{F}_{p}$ and every $r>0$
\begin{align}\label{sup kor so fd}
\frac{1}{r^{\alpha p}}\int_X\int_{B(x,r)}\frac{|f(x)-f(y)|^p}{\mu(B(x,r))}d\mu(y)\,d\mu(x) \le C\liminf_{n \to +\infty} \nu_{n,p}(f,X).
\end{align}
Therefore $\mathcal{F}_{p} \subset KS^{\alpha,p}(X)$ and for every $f \in KS^{\alpha,p}(X)$ and $r>0$
\begin{align}\label{lop gh}
\frac{1}{r^{\alpha p}}\int_X\int_{B(x,r)}\frac{|f(x)-f(y)|^p}{\mu(B(x,r))}d\mu(y)\,d\mu(x) \le C \inf_{f_m} \liminf_{m \to +\infty} \liminf_{n \to +\infty} \nu_{n,p}(f_m,X)
\end{align}
where the infimum is taken over the sequences of functions $f_m \in \mathcal{F}_{p}$ such that $f_m \to f$ in $L^p (X,\mu)$. Note that at this point of the proof, we do not know that for an arbitrary $f \in KS^{\alpha,p}(X)$  there actually exists  a sequence $f_m \in \mathcal{F}_{p}$ such that $f_m \to f$ in $L^p (X,\mu)$ so we do not know yet that the  right-hand side of \eqref{lop gh} is finite; This will be established below by using the controlled cutoff partitions of unity.

Fix  $\eps>0$. Let $\{B_i^\eps=B(x_i,\eps)\}_i$ be an $\eps$-covering of $X$, so that   the family $\{B_i^{5\eps}\}_i$ has the bounded overlap property. Let $\pip_i^\eps$ be a  controlled cutoff partition of unity subordinated to this cover. For $f \in KS^{\alpha,p}(X)$, we  set 
\[
f_\eps:=\sum_i f_{B_i^\eps}\, \pip_i^\eps,
\]
where $f_{B_i^\eps}=\vint_{B_i^\eps} fd\mu$. We first note that $f_\eps \in \mathcal{F}_{p}$ because $\mathcal{F}_{p}$ is a linear space and the above sum is finite since $X$ is compact. We now claim that $f_\eps \to f$ in $L^p(X,\mu)$ when $\eps \to 0$. Indeed, for $x \in B_j^\eps$
\begin{align*}
|f(x)-f_\eps(x)| &\le  \sum_{i:2B_i^\eps\cap 2B_j^\eps\ne\emptyset}|f (x)-f_{B_i^\eps}|\pip_i^\eps(x) \\
 & \le  \sum_{i:2B_i^\eps\cap 2B_j^\eps\ne\emptyset} \left( \fint_{B_i^\eps} | f(x) -f(y)|^p d\mu(y)\right)^{1/p} \\
 & \le C \left( \fint_{B_j^{5\eps}} | f(x) -f(y)|^p d\mu(y)\right)^{1/p}. 
\end{align*}
Therefore we have

\begin{align*}
\int_X |f(x)-f_\eps(x)|^p d\mu(x) & \le C \sum_j \int_{B_j^\eps} \fint_{B_j^{5\eps}} | f(x) -f(y)|^p d\mu(y) d\mu(x) \\
 & \le C  \sum_j \int_{B_j^\eps} \fint_{B(x,6\eps)} | f(x) -f(y)|^p d\mu(y) d\mu(x) \\
 &\le C  \int_{X} \fint_{B(x,6\eps)} | f(x) -f(y)|^p d\mu(y) d\mu(x) \\
 &= C \eps^{p\alpha}  \int_{X} \fint_{B(x,6\eps)}\frac{ | f(x) -f(y)|^p}{\eps^{p\alpha}} d\mu(y) d\mu(x). 
\end{align*}
Therefore, since $f \in KS^{\alpha,p}(X)$, we deduce that $f_\eps \to f$ in $L^p(X,\mu)$.
In particular, $\mathcal{F}_{p}$ is therefore $L^p$-dense in $KS^{\alpha,p}(X)$.  Now, for $x,y\in 2B_j^\eps$ we see that
\begin{align*}
|f_\eps(x)-f_\eps(y)|&=\left|\sum_{i:2B_i^\eps\cap 2B_j^\eps\ne\emptyset}(f_{B_i^\eps}-f_{B_j^\eps})(\pip_i^\eps(x)-\pip_i^\eps(y))\right|\\
&\le \sum_{i:2B_i^\eps\cap 2B_j^\eps\ne\emptyset}|f_{B_i^\eps}-f_{B_j^\eps}||\pip_i^\eps(x)-\pip_i^\eps(y)|.
\end{align*}
Therefore,  we have  that
\[
\int_{B_j^\eps} \int_{2B_j^\eps} |f_\eps(x)-f_\eps(y)|^p  d\rho_n(x, y) \le C \sum_{i:2B_i^\eps\cap 2B_j^\eps\ne\emptyset}|f_{B_i^\eps}-f_{B_j^\eps}|^p \int_{X} \int_{X} |\pip_i^\eps(x)-\pip_i^\eps(y)|^p d\rho_n(x, y) .
\]
Using the locality property of $\rho_n$ one has
\[
\lim_{n \to +\infty} \int_{B_j^\eps} \int_{X \setminus 2B_j^\eps} |f_\eps(x)-f_\eps(y)|^p  d\rho_n(x, y) =0.
\]
We deduce therefore
\begin{align*}
\limsup_{n \to +\infty} \int_{B_j^\eps} \int_{X} |f_\eps(x)-f_\eps(y)|^p  d\rho_n(x, y) \le C \sum_{i:2B_i^\eps\cap 2B_j^\eps\ne\emptyset}|f_{B_i^\eps}-f_{B_j^\eps}|^p \limsup_{n \to +\infty} \int_{X} \int_{X} |\pip_i^\eps(x)-\pip_i^\eps(y)|^p d\rho_n(x, y) .
\end{align*}
From the controlled cutoff condition this yields
\[
\limsup_{n \to +\infty} \int_{B_j^\eps} \int_{X} |f_\eps(x)-f_\eps(y)|^p  d\rho_n(x, y) \le \frac{C}{\eps^{\alpha p}} \sum_{i:2B_i^\eps\cap 2B_j^\eps\ne\emptyset}|f_{B_i^\eps}-f_{B_j^\eps}|^p \mu (B_i^\eps).
\]
Using the same arguments as in the second part of the proof of Theorem \ref{equi poinc} to control the term  $\sum_{i:2B_i^\eps\cap 2B_j^\eps\ne\emptyset}|f_{B_i^\eps}-f_{B_j^\eps}|^p \mu (B_i^\eps)$, we thus see that
\[
\limsup_{n \to +\infty} \int_{B_j^\eps} \int_{X} |f_\eps(x)-f_\eps(y)|^p  d\rho_n(x, y)  \le C   \int_{5B_j^\eps} \vint_{B(x,6\eps)}\frac{|f(y)-f(x)|^p}{\eps^{\alpha p}}\, d\mu(y)\, d\mu(x).
\]
Summing up over $j$ and using the bounded overlap property gives
\[
\limsup_{n \to +\infty} \int_{X} \int_{X} |f_\eps(x)-f_\eps(y)|^p  d\rho_n(x, y)  \le C   \int_{X} \vint_{B(x,6\eps)}\frac{|f(y)-f(x)|^p}{\eps^{\alpha p}}\, d\mu(y)\, d\mu(x).
\]
This implies
\[
\liminf_{\eps \to 0} \liminf_{n \to +\infty} \nu_{n,p}( f_\eps,X)\le C \liminf_{\eps \to 0} \int_{X} \vint_{B(x,\eps)}\frac{|f(y)-f(x)|^p}{\eps^{\alpha p}}\, d\mu(y)\, d\mu(x).
\]
Going back to \eqref{lop gh} we conclude that 
\[
\frac{1}{r^{\alpha p}}\int_X\int_{B(x,r)}\frac{|f(x)-f(y)|^p}{\mu(B(x,r))}d\mu(y)\,d\mu(x)  \le C   \liminf_{\eps \to 0} \int_{X} \vint_{B(x,\eps)}\frac{|f(y)-f(x)|^p}{\eps^{\alpha p}}\, d\mu(y)\, d\mu(x)
\]
and
\[
 \mathrm{Var}_p (f)^p \simeq \inf_{f_m} \liminf_{m \to +\infty} \liminf_{n \to +\infty} \nu_{n,p}(f_m,X).
\]
\end{proof}

To illustrate the scope of our results, we now discuss a situation where the above results can be used when $p=2$. 

\subsubsection*{Generalized Poincar\'e inequalities and controlled cutoffs in Dirichlet spaces}

Let $(X,d,\mu)$ be a compact metric measure space where $\mu$ is a doubling Radon measure. Let $(\mathcal{E}, \mathrm{dom} \, \mathcal{E})$ be a regular and strongly  local Dirichlet form on $L^2(X,\mu)$. Let $P_t$ be the semigroup associated with $\mathcal{E}$ and $p_t(x,dy)$ be the associated heat kernel measures, i.e. for $f \in L^\infty(X,\mu)$,
\[
P_t f(x)=\int_X f(y)p_t(x,dy), \, t>0.
\]
In that setting, one can construct a local transition kernel by considering the family of Radon measures
\[
d\rho_t (x,y)=\frac{1}{t} p_t(y,dx) d\mu (y), \, t>0.
\]
Note that the locality property of $\rho_t$:
\[
\lim_{t \to 0} \frac{1}{t} \int_A \int_B p_t(y,dx) d\mu (y)=0
\]
if $d(A,B)>0$ follows from the assumed locality of the Dirichlet form (see for instance the proof of Proposition 2.1 in \cite{MR2357208}). If $f \in \mathcal{F}_2=\mathrm{dom} \, \mathcal E \cap C(X)$, the measures
\[
\nu_t (f,A) = \frac{1}{t} \int_A \int_X (f(x)-f(y))^2p_t(y,dx) d\mu (y),
\]
converge weakly as $t \to 0$ to the so-called energy measure $d\Gamma(f,f)$ of $f$, see e.g. \cite[(3.2.14)]{MR1303354}. In that setting, the validity of the generalized $2$-Poincar\'e inequality (with $\alpha=d_w/2$, $d_w$ being the walk dimension) and of the controlled cutoff condition can be checked under suitable assumptions by using the results in \cite{MR2228569}. In particular, one gets the following result:

\begin{proposition}\label{appli to}
Assume that the semigroup $P_t$ has a measurable heat kernel $p_t(x,y)$ satisfying, for some $c_{1},c_{2}, c_3, c_4 \in(0,\infty)$ and $ d_h \ge 1, d_{w}\in [2,+\infty)$ the sub-Gaussian estimates:
 \begin{equation}\label{eq:subGauss-upper1}
 c_{1}t^{-d_{h}/d_{w}}\exp\biggl(-c_{2}\Bigl(\frac{d(x,y)^{d_{w}}}{t}\Bigr)^{\frac{1}{d_{w}-1}}\biggr) 
 \le p_{t}(x,y)\leq c_{3}t^{-d_{h}/d_{w}}\exp\biggl(-c_{4}\Bigl(\frac{d(x,y)^{d_{w}}}{t}\Bigr)^{\frac{1}{d_{w}-1}}\biggr)
 \end{equation}
 for $\mu$-a.e. $(x,y)\in X\times X$ and each $t\in\bigl(0, \mathrm{diam}(X) ^{1/d_w} \bigr)$.  Then we have  $KS^{d_w/2,2}(X)= \mathrm{dom} \, \mathcal{E}$. Moreover, there exist  constants $c,C>0$ such that for every $f \in   KS^{d_w/2,2}(X)$,
\[
c\sup_{r >0} E_{2,d_w/2} (f,r) \le \mathcal{E}(f,f) \le C \liminf_{r \to 0} E_{2,d_w/2} (f,r). 
\]
In particular,  the property $\mathcal{P}( 2, d_w/2)$ holds.
\end{proposition}

The identification of the domain of the Dirichlet form as a Besov-Lipschitz space seems to have been first uncovered in \cite{MR1400205}, see also  \cite[Theorem 4.2]{GHL03} and  \cite{MR1800954}. The property $\mathcal{P}(2, d_w/2)$ was first pointed out in  \cite[Proposition 3.5]{MR4196573} where it was proved using completely different methods.


\section{Korevaar-Schoen-Sobolev spaces  and $\mathcal{P}(p,\alpha_p)$  on some simple fractals}\label{section fractal}

The approach to Sobolev spaces based on upper gradients does not work on fractals, see the comments \cite[Page 409]{HKST15} and Remark \ref{comment N} below. Therefore fractals offer an interesting playground to test the scope of the Korevaar-Schoen approach. We  study here  two concrete examples, the Vicsek set and the Sierpi\'nski  gasket,  and prove that those two examples satisfy the property $\mathcal{P}(p,\alpha_p)$ for a critical exponent $\alpha_p>1$.  We moreover prove that the corresponding Korevaar-Schoen-Sobolev spaces are dense in $L^p$. The Vicsek set and the Sierpi\'nski  gasket are examples of  nested fractals (a concept introduced in \cite{MR988082}), and it appears reasonable to infer that our results could be extended to a large class of such nested fractals\footnote{March 2024 update: This extension was recently carried out in \cite{chang2023weak} using similar techniques as we introduce here, see also \cite{gao2023heat}.}. More generally, the theory of Korevaar-Schoen-Sobolev spaces on  p.c.f. fractals seems to be promising to explore in view of the recent results of \cite{CaoGuQiu}.
For different and interesting other approaches to the theory of Sobolev spaces on fractals we  refer the interested reader to \cite{CaoGuQiu}, \cite{MR4175733}, \cite{Kigami}. Those approaches define the Sobolev spaces as the domains of limits of discrete $p$-energies. This is a natural approach in view of the case $p=2$ which yields a rich theory of Dirichlet forms, see \cite{MR1668115} and \cite{MR1840042}.  We explain below how those two approaches coincide on the Vicsek set and the Sierpi\'nski  gasket. It is worth mentioning that  Kigami's general  approach has been recently carried out in great details by Shimizu \cite{Shimizu} for the Sierpi\'nski  carpet. It has been proved there that the domain of the constructed  $p$-energy  turns out to be a Korevaar-Schoen-Sobolev space  when $p$ is larger than the Ahlfors regular conformal dimension of the carpet. The validity or not of the property $\mathcal{P}(p,\alpha_p)$ would be an interesting question to settle  then since the carpet is an example of infinitely ramified fractals, making its geometry very different from the two examples treated below\footnote{March 2024 update: As already pointed out in the introduction, the validity of $\mathcal{P}(p,\alpha_p)$ in the Sierpi\'nski  carpet  was recently proved  for $p$ greater than the Ahlfors regular conformal dimension in  \cite{yang2024korevaarschoen}  and then  shortly after  for every $p >1$ in \cite{murugan2023firstorder}.}.
Finally, let us mention that there also exists a large amount of literature concerning the construction of Sobolev-like and more generally Besov-like functional spaces on fractals using the Laplace operator as a central object, see for instance \cite{MR4347459}, \cite{MR4415965}, \cite{MR2039954}, \cite{MR1962353} and the references therein. When $p \neq 2$, some inclusions are known (following for instance from \cite{BV1}), but making an exact identification  between those spaces and the Korevaar-Schoen ones is a challenging interesting problem  for the future. In a nutshell, to make such connections, it would be interesting to study the possible continuity properties   in the Korevaar-Schoen-Sobolev spaces  of some suitable fractional power of the Laplace operator; This is a boundedness of Riesz transform type problem, see Section 3.5 in \cite{BV3}.

\subsection{Vicsek set}

\begin{figure}[htb]\label{figure1}
 \noindent
 \makebox[\textwidth]{\includegraphics[height=0.22\textwidth]{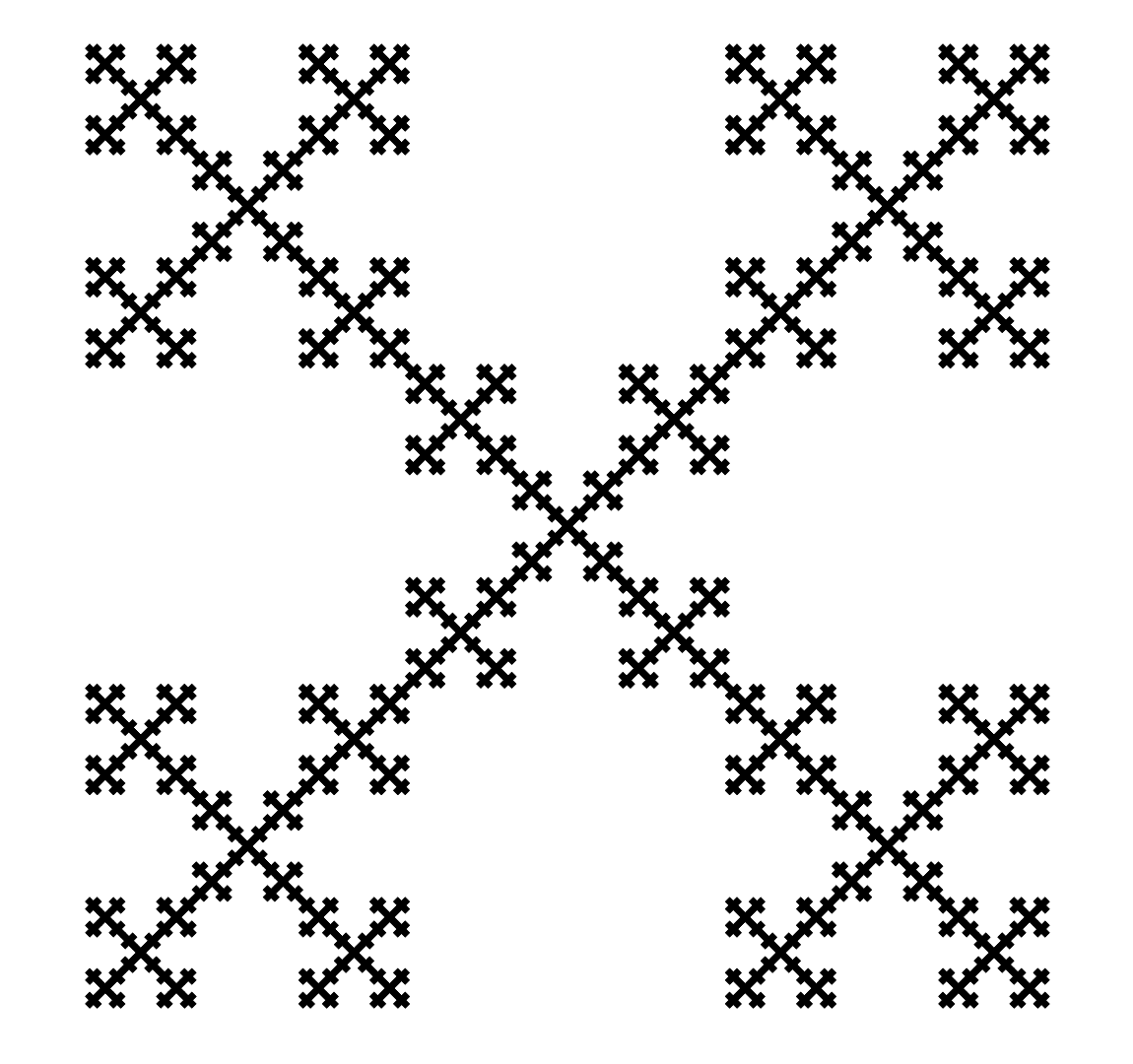}}
  	\caption{Vicsek set}
\end{figure}
Let $q_1=(-\sqrt{2}/2,\sqrt{2}/2)$,$q_2=(\sqrt{2}/2,\sqrt{2}/2)$ , $q_3=(\sqrt{2}/2,-\sqrt{2}/2)$, and $q_4=(-\sqrt{2}/2,-\sqrt{2}/2)$ be the 4 corners of the unit square and let $q_5=(0,0)$ be the center of that square. Define $\psi_i(z)=\frac13(z-q_i)+q_i$ for $1\le i\le 5$. Then the Vicsek set $K$ is the unique non-empty compact set such that 
\[
K=\bigcup_{i=1}^5 \psi_i(K).
\]

\begin{figure}[htb]\label{figure22}
 \noindent
 \makebox[\textwidth]{\includegraphics[height=0.20\textwidth]{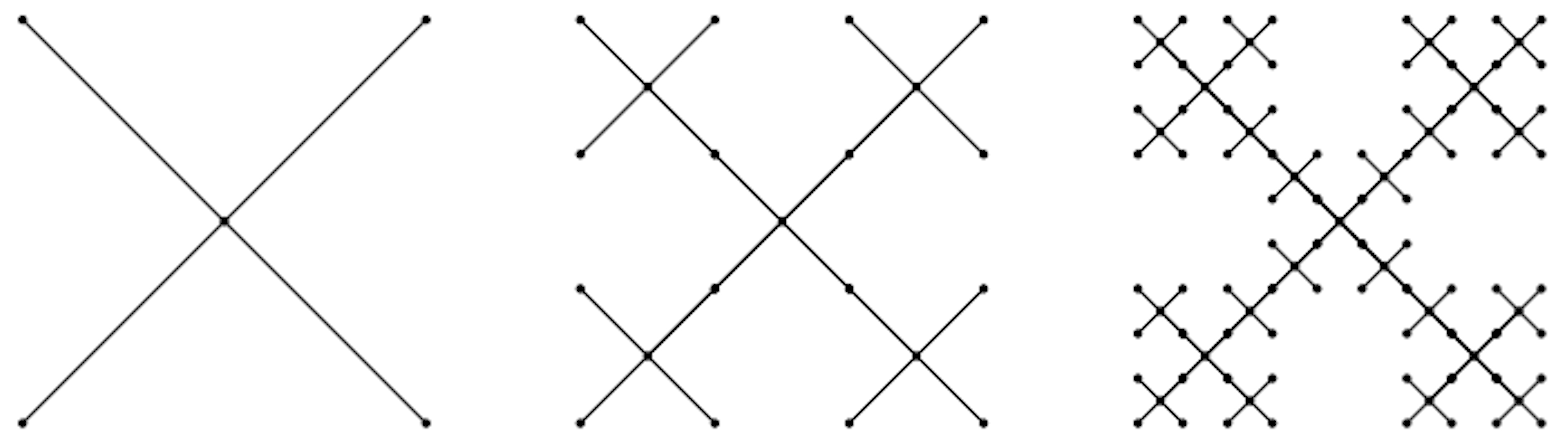}}
  	\caption{Vicsek approximating cable systems $\bar{V}_0$, $\bar{V}_1$ and $\bar{V}_2$}
\end{figure}

Denote $W=\{1, 2,3,4, 5\}$ and $W_n=\{1, 2,3,4, 5\}^n$ for $n\ge 1$. For any $w=(i_1, \cdots, i_n ) \in W_n$, we denote by $\Psi_w$ the contraction mapping $\psi_{i_1}\circ \cdots\circ \psi_{i_n}$ and write $K_w:=\Psi_w(K)$. The set $K_w$ is called  an $n$-simplex.
Let $V_0=\{q_1, q_2, q_3,q_4,q_5\}$. 
We define a sequence of  sets of vertices  $\{V_n\}_{n\ge 0}$ inductively by
\[
V_{n+1}=\bigcup_{i=1}^5 \psi_i(V_n).
\]
Let then $\bar{V}_0$ be the cable system included in $K$ with vertices $V_0=\{q_1,q_2, q_3, q_4,q_5\}$
 and  consider the sequence of cable systems  $\bar{V}_n$ with vertices in $V_n$ inductively defined as follows.  The first cable system is  $\bar{V}_0$ and then 
\[
\bar{V}_{n+1} = \bigcup_{i=1}^5 \psi_i (\bar{V}_n).
\]
Note that $\bar{V}_n \subset K$ and that $K$ is the closure of $\cup_{n \ge 0} \bar{V}_n $. 
The set 
\[
\mathcal{S}= \bigcup_{n \ge 0} \bar{V}_n
\]
is called the skeleton of $K$ and is dense in $K$.   Therefore we have a natural increasing sequence of Vicsek cable systems $\{\bar{V}_n\}_{n\ge 0}$ whose edges have length $3^{-n}$ and whose set of vertices is $V_n$ (see Figure 2). From this viewpoint, the Vicsek set $K$ is seen as a limit of the cable systems $\{\bar{V}_n\}_{n\ge 0}$. 

On $K$ we will consider the  geodesic distance $d$: For $x,y \in \bar{V}_n$, $d(x,y)$ is defined as the length of the geodesic path between $x$ and $y$ and $d$ is then extended by continuity to $K \times K$. The geodesic distance $d$ is then bi-Lipschitz equivalent to the restriction of the Euclidean distance to $K$.  The normalized Hausdorff  measure $\mu$ is the unique Borel measure on $K$ such that for every $i_1, \cdots, i_n \in \{1,2,3, 4,5\}$
\[
\mu(\psi_{i_1} \circ \cdots \circ \psi_{i_n} (K))=5^{-n}.
\]
The Hausdorff dimension of $K$ is then $d_h=\frac{\log 5}{\log 3}$ and the metric space $(K,d)$ is $d_h$-Ahlfors regular in the sense that there exist constants $c,C>0$ such that for every $x \in K$, $r \in (0, \mathrm{diam}\,  K]$,
\[
c \,r^{d_h} \le \mu (B(x,r)) \le Cr^{d_h},
\]
where, as before,  $B(x,r)=\left\{ y \in K: d(x,y) < r \right\}$ denotes the ball with center $x$ and radius $r$ and $\mathrm{diam} \, K=2$ is the diameter of $K$.

\begin{theorem}
For the Vicsek set, for $p \ge 1$ the $L^p$ critical Besov exponent is given by
\[
\alpha_p =1+\frac{d_h-1}{p} >1
\]
and the space $KS^{\alpha_p,p}(K)$ is dense in $L^p(K,\mu)$.
\end{theorem}

\begin{proof}
The case $p=1$ was first treated in \cite[Theorem 5.1]{BV3} where it is actually proved that for any nested fractal $\alpha_1=d_h$ and that the corresponding Korevaar-Schoen-Sobolev space is always dense in $L^1$. The case $1<p \le 2$ was then treated for the first time in \cite[Theorem 3.10]{MR4196573} where it is proved that in that case $\alpha_p =1+\frac{d_h-1}{p}$ and that $KS^{\alpha_p,p}(K)$ contains non-constant functions. The case $p>2$  follows from \cite[Corollary 4.5]{BC2022}.  Finally, the fact that $KS^{\alpha_p,p}(K)$ is dense in $L^p(K,\mu)$ for every $p>1$ follows from Section 2.5 in \cite{BC2022} where it is observed that the set of piecewise affine functions which is dense in $L^p(K,\mu)$, is a subset of $KS^{\alpha_p,p}(K)$.
\end{proof}

Since $\alpha_p > \frac{d_h}{p}$ when $p >1$ it follows from Theorem \ref{T:Morrey_local} that any function $f \in KS^{\alpha_p,p}(K)$, $p>1$, has a version which is $\left(1-\frac{1}{p}\right)$-H\"older continuous. We will therefore see $KS^{\alpha_p,p}(K)$ as a subset of $C(K)$.

For $1 < p <+\infty$, the discrete $p$-energy on $V_m$ of  a function  $f \in C(K)$ is defined as
\[
\mathcal{E}_p^m(f):=\frac{1}{2} 3^{(p-1)m}  \sum_{x,y \in V_m, x \sim y} |f(x)-f(y)|^p.
\]
As a consequence of the basic  inequalities
\[
|x+y+z|^p\le 3^{p-1}(|x|^p+|y|^p+|z|^p), 
\]
and of  the tree structure of $ V_m$ we always have for $p >1$
\begin{equation}\label{eq:increaseEn}
\mathcal{E}_{p}^m(f)\le \mathcal{E}_{p}^n(f), \quad \forall \, m,n\in \mathbb N,m\le n.
\end{equation}
Moreover, from this fact we deduce that   
\begin{equation}\label{eq:lim}
\lim_{n\to \infty}\mathcal{E}_{p}^n(f)=\sup_{n\ge 0} \mathcal{E}_{p}^n(f)=\limsup_{n\to \infty}\mathcal{E}_{p}^n(f)=\liminf_{n\to \infty}\mathcal{E}_{p}^n(f),
\end{equation}
where the above  quantities are in $\mathbb{R}_{\ge 0} \cup \{+\infty \}$.
\begin{defn}
Let $p>1$. For $f \in C(K)$, we define the (possibly infinite) $p$-energy of $f$ by
\[
\mathcal E_{p}(f):=\lim_{m\to \infty} \mathcal E_{p}^m(f)
\]
and let
\[
\mathcal{F}_p=\left\{ f \in C(K): \sup_{m \ge 0} \mathcal{E}_p^m(f) <+\infty \right\}.
\]
We consider on $\mathcal F_p$ the seminorm
\[
\| f \|_{\mathcal{F}_p}= \sup_{m \ge 0} \mathcal{E}_p^m(f)^{1/p}, \, p>1.
\]
\end{defn}

We have then the following characterization of the Korevaar-Schoen-Sobolev spaces in terms of the discrete $p$-energies:

\begin{theorem}\cite[Theorem 2.9]{BC2022} \label{carac Sobolev}
Let $p>1 $. For $f \in C(K)$ the following are equivalent:

\begin{enumerate}
\item $f \in KS^{\alpha_p,p}(K)$;
\item $f \in \mathcal F_p$;
\end{enumerate}
Moreover,  one has
\[
 \sup_{r  > 0} E_{p,\alpha_p} (f,r)^{1/p} \simeq \| f \|_{\mathcal F_p}. 
\]
\end{theorem}

\begin{remark}
In \cite{BC2022} a further characterization of $KS^{\alpha_p,p}(K)$ is given in terms of   weak gradients.
\end{remark}

\begin{remark}\label{comment N}
As in Remark \ref{Newton KS}, one might wonder how the theory of Newtonian Sobolev spaces introduced by Shanmugalingam in \cite{Shanmu00} applies in that setting and if it compares to the Korevaar-Schoen-Sobolev spaces. However, one can see that for every $p \ge 1$, the Newtonian space $N^{1,p}(K)$ is equal to $L^p(K,\mu)$. Indeed, it is well known that there exists a non-negative Borel function $[0,1]\to \mathbb{R}$ whose integral is infinite on any interval of positive length. As a consequence, there exists a non-negative function $\rho$ on the skeleton $\mathcal{S}$ which is measurable with respect to the $\sigma$-algebra generated by the set of geodesic paths included in  $\mathcal{S}$ and such  that for every geodesic path $\gamma$ connecting two points $x\neq y \in \mathcal{S}$, $\int_\gamma \rho =+\infty$. The function $\rho$ can be extended to be zero on $K\setminus \mathcal {S}$. This function $\rho$ is then non-negative and in $L^p(K,\mu)$ for every $p \ge 1$. According to Lemma 2.1 in \cite{Shanmu00} one has therefore $N^{1,p}(K)=L^p(K,\mu)$. Intuitively, the issue is that even though the space is geodesic, the Hausdorff  measure \textit{does not see} the rectifiable paths because they all are supported on a set of measure zero.
\end{remark}

\begin{theorem}\label{P Vicsek}
On the Vicsek set the property $\mathcal{P}(p,\alpha_p)$ holds for every $p \ge 1$. Therefore,  if $p>1$,  $\mathrm{Var}_p (f)^p \simeq \mathcal{E}_p(f)$.
\end{theorem}

\begin{proof}
If $p=1$ the validity of $\mathcal{P}(1,\alpha_1)$ is established in \cite[Theorem 4.9]{BV3}, so we assume that $p>1$. We adapt  slightly  the proof of Theorem 5.4 in \cite{BC2022} and use the notion of piecewise affine function on the Vicsek set.  A continuous function $\Phi:K \to \mathbb{R}$ is called $n$-piecewise affine, if there exists $n \ge 0$ such that $\Phi$ is piecewise affine on the cable system $\bar{V}_n$ (i.e. linear between the vertices of $\bar{V}_n$)  and constant on any connected component of $\bar{V}_m \setminus \bar{V}_n$ for every $m >n$.  If $\Phi:K \to \mathbb{R}$ is an $n$-piecewise affine function, then, for $p >1$, 
\ $\mathcal E_{p}^0 (\Phi) \le \cdots \le \mathcal E_{p}^n (\Phi)=\mathcal E_{p}^m(\Phi)$, where $m \ge n$, and therefore $\mathcal E_{p} (\Phi)=\mathcal E_{p}^n (\Phi)$. 
Let $f \in KS^{\alpha_p,p}(K)$. We define a sequence of piecewise affine functions $\{\Phi_n\}_{n\ge 1}$ associated with $f$ on the cable systems $\{\bar{V}_n\}_{n\ge 1}$ as follows.  For any fixed $n\ge 1$, we first define a function $f_n$ on $V_n$ by
\[
f_{n}(v):=\frac{1}{\mu(K_{n+1}^*(v))} \int_{K_{n+1}^*(v) } f d\mu,\quad v \in V_n,
\]
where $K_{n+1}^*(v)$ is the union of the $(n+1)$-simplices containing $v$. Then let $\Phi_n$ be the unique piecewise affine function that coincides with $f_n$ on $V_n$. It is easy to see that
\[
\Phi_n(x)=\sum_{v \in V_n} \left(\frac{1}{\mu(K_{n+1}^*(v))} \int_{K_{n+1}^*(v) } f d\mu \right) \, u_v(x)
=\sum_{v \in V_n} f_{n}(v) \, u_v(x),
\]
where $u_v$ is the unique $n$-piecewise affine function that takes the value 1 at $v$ and zero on $V_n \setminus \{ v\}$. Note that we have $0\le u_v\le 1$, $\supp u_v \subset K_n^*(v)$, where  $K_n^*(v)$ is the union of the $n$-simplices containing $v$, and 
\[
\sum_{v \in V_n} u_v(x)=1, \quad \forall \, x \in K.
\]
We observe that the covering $\{K_n^*(v)\}_{v\in V_n}$ has the bounded overlap property. Also, for any $x\in K_n^*(v)$,  $K_{n+1}^*(v)\subset B(x, 3^{-n+1})$. We note that $\Phi_n$ is an analogue for the Vicsek set of the sequence $f_\eps$ that was considered in the proof of Theorem \ref{equi poinc}. By H\"older's inequality one has:
\begin{equation}\label{eq:g-Lp}
\begin{split}
   \|f -\Phi_n \|_{L^p(K,\mu)}^p 
   &\le C\sum_{v\in V_n} \int_{K_n^*(v)}|f(x)-f_n(v)|^p (u_v(x))^p d\mu(x)
   \\ &\le C\sum_{v\in V_n} \int_{K_n^*(v)}\fint_{K_{n+1}^*(v)} |f(x)-f(y)|^p d\mu(y) d\mu(x)
   \\ &\le C\int_{K}\fint_{B(x, 3^{-n+1})} |f(x)-f(y)|^p d\mu(y) d\mu(x).
\end{split}
\end{equation}
On the other hand from Theorem \ref{carac Sobolev}, we have for every $r >0$
\[
\frac{1}{r^{p\alpha_p}}\int_{K}\fint_{B(z, r)}|\Phi_n(z)-\Phi_n(w)|^pd\mu(w)d\mu(z) \le C \mathcal E_{p}(\Phi_n)=C \mathcal E_p^n(\Phi_n)
\]
and $\mathcal E_p^n(\Phi_n)$ can be controlled as follows. Observe that for any $x\in V_n$, one has $\Phi_n(x)=f_n(x)$ by definition. Hence 
\begin{align*}
    \mathcal E_p^n(\Phi_n)=\frac123^{(p-1)n}\sum_{x,y\in V_n,x\sim y}|f_n(x)-f_n(y)|^p.
\end{align*}
For any neighboring vertices $x,y \in V_n$, H\"older's inequality yields
\begin{align*}
    |f_n(x)-f_n(y)|
    & \le \frac{1}{\mu(K_{n+1}^*(x))\mu(K_{n+1}^*(y))}\int_{K_{n+1}^*(x)}\int_{K_{n+1}^*(y)}|f(z)-f(w)|d\mu(w)d\mu(z)
    \\ &\le C\left( 5^{2n} \int_{K_{n+1}^*(x)}\int_{K_{n+1}^*(y)}|f(z)-f(w)|^pd\mu(w)d\mu(z)\right)^{\frac1p}.
\end{align*}
Thanks to the fact that $x,y\in V_n$ are adjacent $K_{n+1}^*(y) \subset B(z,  3^{-n+1})$ for any $z\in K_{n+1}^*(x)$. Therefore we get
\[
  |f_n(x)-f_n(y)|^p\le C 5^{2n} \int_{K_{n+1}^*(x)}\int_{B(z,  3^{-n+1})}|f(z)-f(w)|^pd\mu(w)d\mu(z).
\]
By the bounded overlap property of $\{K_{n+1}^*(v)\}_{v\in V_n}$, we then have
\begin{align*}
    \mathcal E_p^n(\Phi_n)
    & \le C 3^{(p-1)n}  5^{2n} \sum_{x,y\in V_n, x\sim y} \int_{K_{n+1}^*(x)}\int_{K_{n+1}^*(y)}|f(z)-f(w)|^pd\mu(w)d\mu(z)
    \\ &\le  C 3^{(p-1)n} 5^{2n} \int_{K}\int_{B(z,  3^{-n+1})}|f(z)-f(w)|^pd\mu(w)d\mu(z).
\end{align*}
Set $r_n= 3^{-n+1}$. We can rewrite the above inequality as
\[
\mathcal E_p^n(\Phi_n) \le  \frac{C}{r_n^{p\alpha_p}}\int_{K}\fint_{B(z, r_n)}|f(z)-f(w)|^pd\mu(w)d\mu(z).
\]
Consequently, we have for every $r >0$
\[
\frac{1}{r^{p\alpha_p}}\int_{K}\fint_{B(z, r)}|\Phi_n(z)-\Phi_n(w)|^pd\mu(w)d\mu(z) \le \frac{C}{r_n^{p\alpha_p}}\int_{K}\fint_{B(z, r_n)}|f(z)-f(w)|^pd\mu(w)d\mu(z).
\]
Let now $0< \eps <1$ and let  $n_\eps$ be the unique integer $\ge 1$ such that $\frac{1}{3^{n_\eps-1}} < \eps \le \frac{1}{3^{n_\eps-2}}$. We have for every $r >0$
\begin{align}\label{lim almost}
\frac{1}{r^{p\alpha_p}}\int_{K}\fint_{B(z, r)}|\Phi_{n_\eps}(z)-\Phi_{n_\eps}(w)|^pd\mu(w)d\mu(z) \le \frac{C}{\eps^{p\alpha_p}}\int_{K}\fint_{B(z, \eps)}|f(z)-f(w)|^pd\mu(w)d\mu(z).
\end{align}
However, \eqref{eq:g-Lp}  also gives that 
\[
\|f -\Phi_{n_\eps} \|_{L^p(K,\mu)}^p\le C \eps^{p\alpha_p}  \frac{1}{\eps^{p\alpha_p}}\int_{K}\fint_{B(z, \eps)}|f(z)-f(w)|^pd\mu(w)d\mu(z)
\]
which implies that as $\eps\to 0$, $\Phi_{n_\eps} \to f$  in $L^p(K,\mu)$. By taking $\liminf_{\eps \to 0}$ in \eqref{lim almost} we obtain then that for every $r>0$
\[
\frac{1}{r^{p\alpha_p}}\int_{K}\fint_{B(z, r)}|f(z)-f(w)|^pd\mu(w)d\mu(z) \le C \mathrm{Var}_p (f)^p.
\]
\end{proof}

\begin{remark}
In view of the property $\mathcal{P}(p,\alpha_p)$ on the Vicsek set, one might wonder (as in Remark \ref{existence limit poincare}) whether the limit
\[
\lim_{\eps \to 0} \frac{1}{\eps^{p\alpha_p}}\int_{K}\fint_{B(z, \eps)}|f(z)-f(w)|^pd\mu(w)d\mu(z)
\]
actually exists or not. This problem was studied in \cite{AB2021} in the case $p=1$ and it was proved that the limit  does not exist in general due to the small-scale oscillations appearing in the geometry of the Vicsek set.
\end{remark}

As a  consequence of the property $\mathcal{P} (p,\alpha_p)$ and Theorem \ref{sobo}, we therefore get the following Nash inequalities on the Vicsek set.

\begin{cor}\label{Nash Vicsek} 
For $p>1$ the following Nash inequality holds for every $f \in KS^{\alpha_p,p}(K)$,
\[
\| f \|_{L^p(K,\mu)} \le C \left(\|f\|_{L^p(K,\mu)}+ \mathrm{Var}_p (f)\right)^{\theta} \| f \|^{1-\theta}_{L^1(K,\mu)}
\]
with $\theta=\frac{(p-1)d_h}{p-1+pd_h}$, while for $p=1$, there exists a constant $C>0$ such that for every $f \in KS^{d_h,1}(K)$
\begin{align*}
\| f \|_{L^\infty(K,\mu)} \le C \left(\|f\|_{L^1(K,\mu)}+ \mathrm{Var}_1 (f)\right).
\end{align*}
\end{cor}

\begin{remark}
For $p=2$, Nash inequalities  have been studied in connection with heat kernel estimates in the more general context of p.c.f. fractals. Corollary \ref{Nash Vicsek}  recovers the special case of  \cite[Theorem 8.3]{MR1668115} for the Vicsek set since  $\mathrm{Var}_2 (f)^2 \simeq \mathcal{E}_2(f)$ which is the Dirichlet form on the Vicsek set. 
\end{remark}

\subsection{Sierpi\'nski  gasket}\label{SG section}

Let $V_0=\{q_1, q_2, q_3\}$ be the set of vertices of an equilateral triangle of side 1 in $\mathbb C$. Define 
$$\psi_i(z)=\frac{z-q_i}{2}+q_i$$ for $i=1,2,3$. Then the Sierpi\'nski  gasket $K$ is the unique non-empty compact subset in $\mathbb C$ such that 
\[
K=\bigcup_{i=1}^3 \psi_i(K).
\]

The Hausdorff dimension of $K$ with respect to the geodesic metric $d$ is given by $d_h=\frac{\log 3}{\log 2}$. The (normalized) Hausdorff measure on $K$ is given by the unique Borel measure $\mu$ on $K$ such that for every $i_1, \cdots, i_n \in \{ 1,2,3 \} $,
\[
\mu \left(  \psi_{i_1} \circ \cdots \circ \psi_{i_n}  (K)\right)=3^{-n}.
\]

\begin{figure}[htb]\centering
	\includegraphics[trim={60 10 180 60},height=0.25\textwidth]{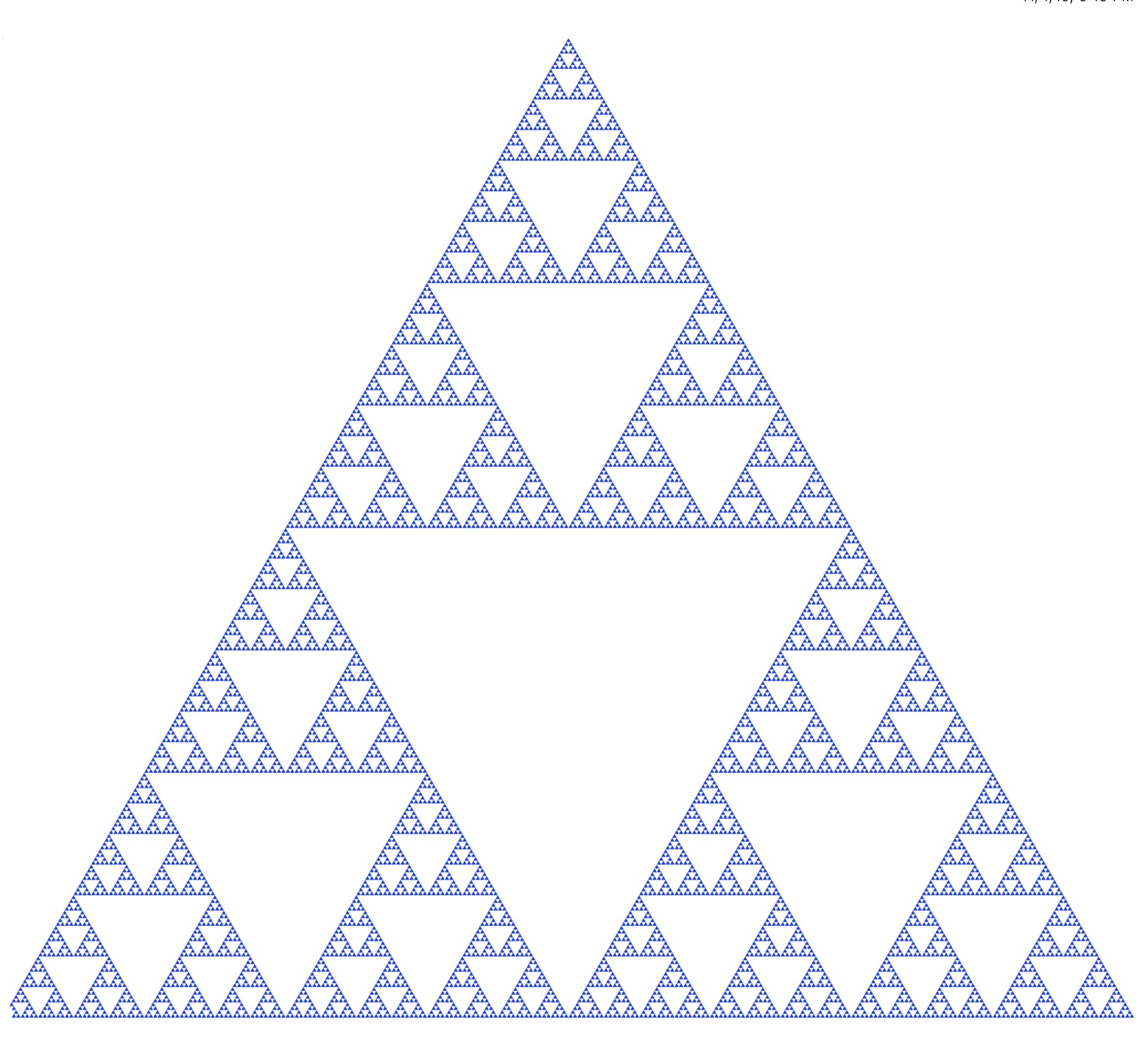}
	\caption{Sierpi\'nski gasket.} \label{fig-SG}
\end{figure}

This measure $\mu$ is $d_h$-Ahlfors regular, i.e. there exist constants $c,C>0$ such that for every $x \in K$ and $r \in (0, \mathrm{diam} (K) ]$,
\begin{equation}
\label{Ahlfors}
c r^{d_h} \le \mu (B(x,r)) \le C r^{d_h}.
\end{equation}

We define a sequence of sets $\{V_m\}_{m\ge 0}$ inductively by
\[
V_{m+1}=\bigcup_{i=1}^3 \psi_i(V_m).
\]

Let $p>1$. The study of $p$-energies on the Sierpi\'nski  gasket was undertaken in \cite{HPS} where the authors introduced a non-linear renormalization problem from which it is possible to compute the $L^p$ critical Besov exponents. After \cite{HPS} consider the function
\[
F_p (a)=|a_1-a_2|^p+|a_2-a_3|^p+|a_3-a_1|^p, \quad a=(a_1,a_2,a_3) \in \mathbb R^3.
\]
The renormalization problem for the function $F_p$ is to find a non-negative continuous convex  function $A_p$ on  $\mathbb R^3$ and a number $0<r_p<1$ such that:
\begin{enumerate}
\item $A_p(a) \simeq F_p(a)$;
\item $\min_{b \in \mathbb R^3} \left(A_p(a_1,b_2,b_3)+A_p(b_1,a_2,b_3) +A_p(b_1,b_2,a_3)\right)=r_p A_p(a)$.
\end{enumerate}
It was shown in \cite{HPS} that  a solution of the renormalization problem exists. While the uniqueness of the function $A_p$ is not  known, remarkably the number $r_p$ is. For $p=2$ the value of $r_p$ is $\frac{3}{5}$. For other values of $p$, the value of $r_p$ is unknown but was proved in \cite{HPS} to satisfy
\[
2^{1-p} \le r_p  \le 2^{p-1} \left( 1+\sqrt{1+2^{3-1/(p-1)}} \right)^{1-p} < 3  \cdot 2^{-p}.
\]
\begin{theorem}
For the Sierpi\'nski  gasket, if $p > 1$ the $L^p$ critical Besov exponent is given by
\[
\alpha_p =\frac{1}{p}\left(  \frac{\log 3}{\log 2}-\frac{\log r_p}{\log 2}\right)  \in \left( \frac{d_h}{p}, 1+\frac{d_h-1}{p}  \right]
\]
and for $p=1$,
\[
\alpha_1=d_h=\frac{\log 3}{\log 2}.
\] 
Moreover, for every $p \ge 1$, the Korevaar-Schoen-Sobolev space $KS^{\alpha_p,p}(K)$ is dense in $L^p(K,\mu)$.
\end{theorem}

\begin{proof}
As before, the case $p=1$ follows from \cite[Theorem 5.1]{BV3}. For $p>1$ the result follows from the proof of Theorem 2.1 in  \cite{HuJiWen}. Indeed, a thorough inspection shows that it is actually proved there that for $\alpha=\frac{1}{p}\left( \frac{\log 3}{\log 2}-\frac{\log r_p}{\log 2}\right)$ one has
\[
\sup_{r >0} E_{p,\alpha}(f,r) \le C \limsup_{r  \to 0} E_{p,\alpha}(f,r).
\]
See also Corollary 2.1 in the same paper \cite{HuJiWen}. Finally, $KS^{\alpha_p,p}(K)$ is dense in $L^p(K,\mu)$ from Lemma \ref{approx} below.
\end{proof}

Since $\alpha_p > \frac{d_h}{p}$ when $p >1$ it follows from Theorem \ref{T:Morrey_local} that any function $f \in KS^{\alpha_p,p}(K)$, $p>1$, has a version which is $\alpha_p-\frac{d_h}{p}=-\frac{\log r_p}{p \log 2}$-H\"older continuous, see Remark \ref{regular version}. We will therefore see $KS^{\alpha_p,p}(K)$ as a subset of $C(K)$ when $p>1$.

As for the Vicsek set, it is possible to characterize the Korevaar-Schoen-Sobolev space as the domain of a limit of discrete $p$-energies. Denote $W=\{1, 2,3\}$ and $W_n=\{1, 2,3\}^n$ for $n\ge 1$. For any $w=(i_1, \cdots, i_n) \in W_n$, we denote by $\Psi_w$ the contraction mapping $\psi_{i_1}\circ \cdots\circ \psi_{i_n}$ and write $K_w:=\Psi_w(K)$. As before, the set $K_w$ is called  an $n$-simplex.

For $1 < p <+\infty$, the discrete $p$-energy on $V_m$ of  a function  $f \in C(K)$ is defined as
\[
\mathcal{E}_p^m(f):=\frac{1}{2} r_p^{-m}  \sum_{x,y \in V_m, x \sim y} |f(x)-f(y)|^p=r_p^{-m} \sum_{w \in W_m} F_p (f(\Psi_w(q_1)),f(\Psi_w(q_2)),f(\Psi_w(q_3))).
\]
Unlike for the Vicsek set, the sequence $\mathcal{E}_p^m(f)$ needs not be non-decreasing, however one may instead consider the modified $p$-energy
\[
\mathcal{A}_p^m(f):=r_p^{-m}\sum_{w \in W_m} A_p (f(\Psi_w(q_1)),f(\Psi_w(q_2)),f(\Psi_w(q_3)))
\]
where $A_p$ solves the renormalization problem. It  is then clear that for every $m$
\[
c  \mathcal{A}_p^m(f) \le \mathcal{E}_p^m(f)\le C \mathcal{A}_p^m(f)
\]
and moreover that $\mathcal{A}_p^m(f)$ is non-decreasing. In particular $\sup_{m \ge 0} \mathcal E_{p}^m(f)$ is finite if and only if the limit  $\lim_{m \to +\infty} \mathcal{A}_p^m(f) $ is finite.

\begin{defn}
Let $p>1$. For $f \in C(K)$, we define the (possibly infinite) $p$-energy of $f$ by
\[
\mathcal E_{p}(f):= \sup_{m \ge 0} \mathcal E_{p}^m(f).
\]
We define then
\[
\mathcal{F}_p=\left\{ f \in C(K): \sup_{m \ge 0} \mathcal{E}_p^m(f) <+\infty \right\}
\]
and consider on $\mathcal F_p$ the seminorm
\[
\| f \|_{\mathcal{F}_p}= \sup_{m \ge 0} \mathcal{E}_p^m(f)^{1/p}.
\]
\end{defn}

We have then the following characterization of the Korevaar-Schoen-Sobolev spaces in terms of the discrete $p$-energies:

\begin{theorem}\cite[Theorem 2.1]{HuJiWen} \label{carac Sobolev 2}
Let $p>1 $. For $f \in C(K)$ the following are equivalent:

\begin{enumerate}
\item $f \in KS^{\alpha_p,p}(K)$;
\item $f \in \mathcal F_p$;
\end{enumerate}
Moreover,  one has
\[
 \sup_{r  > 0} E_{p,\alpha_p} (f,r)^{1/p} \simeq \| f \|_{\mathcal F_p}. 
\]
\end{theorem}

One can construct plenty of functions in $\mathcal F_p$ and therefore in $KS^{\alpha_p,p}(K)$ by using the notion of $p$-harmonic extension. This extension procedure is explained in detail  in Corollary 2.4 in \cite{HPS} (see also \cite{CaoGuQiu}) and can be described as follows.  Let $n \ge 0$ and $f_n:V_n \to \mathbb R$. One can  extend $f_{n}$ into a function $f_{n+1}$ defined on $V_{n+1}$ such that:
\begin{itemize}
\item  For all $v \in V_n$, $f_{n+1}(v)=f_n(v)$;
\item $\mathcal{A}_p^{n+1}(f_{n+1})= \min \{ \mathcal{A}_p^{n+1}(g): g_{\mid V_n}=f_n \}$;
\end{itemize}
Indeed, if we fix  $w \in W_n$ and minimize over $b_1,b_2,b_3$ the quantity
\[
A_p (f_n(\Psi_w(q_1)),b_2,b_3)+A_p (b_1,f_n(\Psi_w(q_2)),b_3)+A_p(b_1,b_2,f_n(\Psi_w(q_3)))
\]
a minimizer will assign values of $f_{n+1}$ on $\Psi_w (V_1)$.
We note  then that
\[
\mathcal{A}_p^{n+1}(f_{n+1})=\mathcal{A}_p^{n}(f_{n}).
\]
The process can be repeated and we thus get a sequence of functions $f_m :V_m \to \mathbb R$, $m \ge n$ such that for every $v \in V_m$, $f_{m+1}(v)=f_m(v)$ and
\[
\mathcal{A}_p^{m}(f_{m})=\mathcal{A}_p^{n}(f_{n}).
\]
We denote then by $H_p(f_n)$ the function on $\cup_m V_m$ whose  restriction to each $V_m$ coincides with $f_m$. The function $H_p(f_n)$ is called a $p$-harmonic extension of $f_n$. We note that for suitable choices of $A_p$ the $p$-harmonic extension of $f_n$ is unique, see the discussion before Lemma A.2. in \cite{CaoGuQiu}.

\begin{lem}\label{esti p}
There exists a constant $C>0$ such that for every $u,v \in \cup_m V_m$,
\[
|H_p(f_n)(u)-H_p(f_n)(v)|^p \le  C d(u,v)^{-\frac{\log r_p}{\log 2}} \mathcal{A}_p^n(f_n).
\]
In particular $H_p(f_n)$ has a unique H\"older continuous extension to $K$, which we still denote by $H_p(f_n)$.
\end{lem}

\begin{proof}
It is clear that for $m \ge n$ and $ u, v \in V_m, u \sim v$
\[
| f_m(u)-f_m(v)|^p \le r_p^m \mathcal{E}_p^m(f) \le Cr_p^m \mathcal{A}_p^m(f)=Cr_p^m \mathcal{A}_p^n(f).
\]
Therefore,
\[
 |H_p(f_n)(u)-H_p(f_n)(v)|^p \le  Cr_p^m \mathcal{A}_p^n(f).
 \]
 Since $d(u,v)=\frac{1}{2^m}$ this can be rewritten as
 \[
 |H_p(f_n)(u)-H_p(f_n)(v)|^p \le  C d(u,v)^{-\frac{\log r_p}{\log 2}} \mathcal{A}_p^n(f).
 \]
 Now, for general $u,v \in V_m$ one can use a chaining argument similar to the one described in the third paragraph of
Section 1.4 in \cite{Str06}, see also \cite[Proof of Proposition 5.3 (d)]{CaoGuQiu} .
\end{proof}

The following result follows  from Corollary 2.4 in \cite{HPS}.

\begin{lem}\label{approx}
Let $p >1$. For every $f \in C(K)$, let $f_n:V_n \to \mathbb R$ be the unique function on $V_n$ that coincides with $f$. Then
\[
\lim_{n\to +\infty} \sup_{x \in K} | H_p(f_n) (x) -f(x) |=0.
\]

\end{lem}
%
%
%

We are now ready to prove the following:

\begin{theorem}
Let $p \ge 1$. On the Sierpi\'nski  gasket the property $\mathcal{P}(p,\alpha_p)$ holds. Therefore,  if $p>1$,  $\mathrm{Var}_p (f)^p \simeq \mathcal{E}_p(f)$.
\end{theorem}

\begin{proof}
If $p=1$ the validity of $\mathcal{P}(1,\alpha_1)$ is established in \cite[Theorem 4.9]{BV3}, so we assume that $p>1$. The proof is relatively similar to the proof of Theorem \ref{P Vicsek}; the idea is to replace piecewise affine functions by $p$-harmonic extensions. Let $f \in KS^{\alpha_p,p}(K)$. For any fixed $n\ge 1$, we first define a function $\hat f_n$ on $V_n$ by
\[
\hat f_{n}(v):=\frac{1}{\mu(K_{n+1}^*(v))} \int_{K_{n+1}^*(v) } f d\mu,\quad v \in V_n,
\]
where $K_{n+1}^*(v)$ is the union of the $(n+1)$-simplices containing $v$. Then, we let $\Phi_n=H_p(\hat f_n)$. We observe that the covering $\{K_n^*(v)\}_{v\in V_n}$ has the bounded overlap property. Also, for any $x\in K_n^*(v)$,  $K_{n+1}^*(v)\subset B(x, 2^{-n+1})$.  From Theorem \ref{carac Sobolev 2}, we have for every $r >0$
\[
\frac{1}{r^{p\alpha_p}}\int_{K}\fint_{B(z, r)}|\Phi_n(z)-\Phi_n(w)|^pd\mu(w)d\mu(z) \le C \mathcal E_{p}(\Phi_n) \le C\mathcal A_{p}(\Phi_n)   \le C \mathcal E_p^n(\Phi_n)
\]
and $\mathcal E_p^n(\Phi_n)$ can be controlled as follows. Observe that for any $x\in V_n$, one has $\Phi_n(x)=\hat f_n(x)$ by definition. Hence 
\begin{align*}
    \mathcal E_p^n(\Phi_n)=\frac12 r_p^{-n}\sum_{x,y\in V_n,x\sim y}|\hat f_n(x)-\hat f_n(y)|^p.
\end{align*}
For any neighboring vertices $x,y \in V_n$, H\"older's inequality yields
\begin{align*}
    |\hat f_n(x)-\hat f_n(y)|
    & \le \frac{1}{\mu(K_{n+1}^*(x))\mu(K_{n+1}^*(y))}\int_{K_{n+1}^*(x)}\int_{K_{n+1}^*(y)}|f(z)-f(w)|d\mu(w)d\mu(z)
    \\ &\le C\left(3^{2n} \int_{K_{n+1}^*(x)}\int_{K_{n+1}^*(y)}|f(z)-f(w)|^pd\mu(w)d\mu(z)\right)^{\frac1p}.
\end{align*}
Notice   that if $x,y\in V_n$ are adjacent then $K_{n+1}^*(y) \subset B(z,  2^{-n+2})$ for any $z\in K_{n+1}^*(x)$. Therefore we get
\[
  |\hat f_n(x)-\hat f_n(y)|^p\le C 3^{2n}\int_{K_{n+1}^*(x)}\int_{ B(z,  2^{-n+2})}|f(z)-f(w)|^pd\mu(w)d\mu(z).
\]
By the bounded overlap property of $\{K_{n+1}^*(v)\}_{v\in V_n}$, we then have
\begin{align*}
    \mathcal E_p^n(\Phi_n)
    & \le Cr_p^{-n}3^{2n} \sum_{x,y\in V_n, x\sim y} \int_{K_{n+1}^*(x)}\int_{K_{n+1}^*(y)}|f(z)-f(w)|^pd\mu(w)d\mu(z)
    \\ &\le  Cr_p^{-n}3^{2n} \int_{K}\int_{ B(z,  2^{-n+2})}|f(z)-f(w)|^pd\mu(w)d\mu(z).
\end{align*}
Set $\eps_n= 2^{-n+2}$. We can rewrite the above inequality as
\begin{align}\label{esti phi}
\mathcal E_p^n(\Phi_n) \le  \frac{C}{\eps_n^{p\alpha_p}}\int_{K}\fint_{ B(z, \eps_n)}|f(z)-f(w)|^pd\mu(w)d\mu(z).
\end{align}
Consequently, we have for every $r >0$
\[
\frac{1}{r^{p\alpha_p}}\int_{K}\fint_{B(z, r)}|\Phi_n(z)-\Phi_n(w)|^pd\mu(w)d\mu(z) \le \frac{C}{\eps_n^{p\alpha_p}}\int_{K}\fint_{ B(z, \eps_n)}|f(z)-f(w)|^pd\mu(w)d\mu(z).
\]

To conclude, it remains therefore to prove that $\Phi_n$ converges to $f$ in $L^p(K,\mu)$. Since $f \in KS^{\alpha_p,p}(K)$, from  Theorem \ref{T:Morrey_local}, for every $x,y \in K$
\[
|  f(x) -f(y) |^p \le C d(x,y)^{-\frac{\log r_p}{\log 2}}.
\]
Let now $n \ge 1$ and $w \in W_n$. One has for $x \in K_w$
\[
| f(x) - \Phi_n (x)|\le | f(x) -\hat f_n(v)| + |\Phi_n (v)-\Phi_n (x)|  
\]
where $v$ is a vertex of $K_w$.  We have first
\begin{align*}
| f(x) -\hat f_n(v)| & = \left| f(x) -\frac{1}{\mu(K_{n+1}^*(v))} \int_{K_{n+1}^*(v) } f d\mu \right| \\
 &\le  C 3^n  \int_{K_{n+1}^*(v) }| f(x) -f(y)| d\mu(y) \\
 & \le C r_p^n.
\end{align*}
Then, from Lemma \ref{esti p}
\begin{align*}
|\Phi_n (v)-\Phi_n (x)| \le C r_p^n \mathcal{A}^n_p( \Phi_n) \le C r_p^n \mathcal{E}^n_p( \Phi_n) .
\end{align*}
From \eqref{esti phi} we know that 
\[
\mathcal{E}^n_p( \Phi_n) \le C \sup_{\eps >0} \frac{1}{\eps^{p\alpha_p}}\int_{K}\fint_{B(z, \eps)}|f(z)-f(w)|^pd\mu(w)d\mu(z) <+\infty.
\]
Therefore, for all $x \in K$
\[
| f(x) - \Phi_n (x)|\le C r_p^n,
\]
which implies that $\Phi_n \to f$ in $L^p(K,\mu)$. The proof is by now complete.
\end{proof}

As a  consequence of the property $\mathcal{P} (p,\alpha_p)$ and Theorem \ref{sobo}, we  get the following Nash inequalities on the Sierpi\'nski  gasket.

\begin{cor}
For $p>1$ the following Nash inequality holds for every $f \in KS^{\alpha_p,p}(K)$,
\[
\| f \|_{L^p(K,\mu)} \le C \left(\|f\|_{L^p(K,\mu)}+ \mathrm{Var}_p (f)\right)^{\theta} \| f \|^{1-\theta}_{L^1(K,\mu)}
\]
with $\theta=\frac{(p-1)d_h}{p(\alpha_p+d_h)-d_h}$, while for $p=1$ there exists a constant $C>0$ such that for every $f \in KS^{d_h,1}(K)$
\begin{align*}
\| f \|_{L^\infty(K,\mu)} \le C \left(\|f\|_{L^1(K,\mu)}+ \mathrm{Var}_1 (f)\right).
\end{align*}
\end{cor}

\section{Korevaar-Schoen-Sobolev spaces and  heat kernels}

In this section, for completeness, we now briefly survey some of the results in \cite{MR4196573} and \cite{BV2,BV3,BV1}  where a connection was deepened between the theory of Dirichlet forms and the theory of Korevaar-Schoen-Sobolev spaces following earlier works like \cite{MR2574998}. This connection allows to study properties of the Korevaar-Schoen-Sobolev spaces in some settings  where Poincar\'e inequalities might not be available but a rich theory of heat kernels is. We also mention several open questions related to this approach which are connected to the results of the present paper.

\subsection{Dirichlet forms with Gaussian or sub-Gaussian heat kernel estimates}

As before, $(X,d,\mu)$ is a metric measure space where $\mu$ is a positive and doubling Borel regular measure. We use the basic definitions and properties of Dirichlet forms and associated  heat semigroups listed in \cite[Section 2]{BV1}.  For a complete exposition of the theory we refer to \cite{MR1303354}.

 Let  $(\mathcal{E},\mathcal{F}=\mathbf{dom}(\mathcal{E}))$ be a Dirichlet form on $L^2(X,\mu)$. We call $(X,d,\mu,\mathcal E, \mathcal F)$ a metric measure Dirichlet space. We  assume that the semigroup $\{P_t\}$ associated with $\mathcal{E}$ is stochastically complete (i.e. $P_t 1=1$) and  has a measurable heat kernel $p_t(x,y)$ satisfying, for some $c_{1},c_{2}, c_3, c_4 \in(0,\infty)$ and $ d_h \ge 1, d_{w}\in [2,+\infty)$,
 \begin{equation}\label{eq:subGauss-upper}
 c_{1}t^{-d_{h}/d_{w}}\exp\biggl(-c_{2}\Bigl(\frac{d(x,y)^{d_{w}}}{t}\Bigr)^{\frac{1}{d_{w}-1}}\biggr) 
 \le p_{t}(x,y)\leq c_{3}t^{-d_{h}/d_{w}}\exp\biggl(-c_{4}\Bigl(\frac{d(x,y)^{d_{w}}}{t}\Bigr)^{\frac{1}{d_{w}-1}}\biggr)
 \end{equation}
 for $\mu$-a.e. $(x,y)\in X\times X$ for each $t\in\bigl(0,\mathrm{diam}(X)^{d_w}\bigr)$.  As before, $\mathrm{diam}(X)$ is the diameter of $X$ which could possibly  be $+\infty$.
 
 The exact values of $c_1,c_2,c_3,c_4$ are irrelevant.  However, the parameters $d_h$ and $d_w$ are important.  We will see below that the parameter $d_h$ is the Ahlfors dimension (volume exponent). The parameter $d_w$ is  called the walk dimension (for its probabilistic interpretation). When $d_w=2$, one speaks of Gaussian estimates and when $d_w > 2$, one speaks of sub-Gaussian estimates. 
 
 In some concrete situations like manifolds or fractals, the estimates \eqref{eq:subGauss-upper} might be obtained using geometric, analytic or probabilistic methods. A large amount of literature is  devoted to the study of such estimates, see for instance \cite{MR2569498}, \cite{MR3276002}, \cite{MR3417504},  \cite{MR2512802}. Therefore, at least for our purpose here, they are a reasonable  assumption to work with. In Barlow \cite{MR1668115}, geodesic complete metric spaces supporting a heat kernel satisfying the estimates \eqref{eq:subGauss-upper} are called fractional spaces.
 
 A basic consequence of \eqref{eq:subGauss-upper} is  the $d_h$-Ahlfors regularity of the space (see \cite[Theorem 3.2]{GHL03}): There exist constants $c,C>0$ such that for every $x \in X$, $r \in (0, \mathrm{diam} (X) )$,
\[
c r^{d_h} \le  \mu (B(x,r)) \le C r^{d_h}.
\]

The first important result connecting the theory of Dirichlet forms to the theory of Korevaar-Schoen-Sobolev spaces is the following result, which was already mentioned and reproved in the case where $X$ is compact, see Proposition \ref{appli to}.  
\begin{theorem}
We have $\mathcal F=KS^{d_w/2,2}(X)$. Moreover, there exist  constants $c,C>0$ such that for every $f \in   KS^{d_w/2,2}(X)$,
\[
c\sup_{r >0} E_{2,d_w/2} (f,r) \le \mathcal{E}(f,f) \le C \liminf_{r \to 0} E_{2,d_w/2} (f,r). 
\]
In particular, the property $\mathcal{P}( 2, d_w/2)$ holds.
\end{theorem}

Note that in   \cite[Theorem 4.1]{MR2161694} the domain of the Dirichlet form associated with some fractals satisfying  more general heat kernel estimates is identified as a Besov-Lipschitz space with variable regularity.  In that framework it would be interesting to investigate a theory of Korevaar-Schoen-Sobolev spaces with variable regularity.

\subsection{Besov-Lipschitz spaces and heat kernels}

 Let $p \ge 1$ and $\beta \ge 0$. Similarly to \cite{BV1}, we define the  Besov seminorm associated with the heat semigroup as follows
\begin{equation}\label{eq:defnofheatbesovnorm}
\| f \|_{p,\beta }:= \sup_{t \in (0,\mathrm{diam}(X)^{d_w})} t^{-\beta} \left( \int_X \int_X |f(x)-f(y) |^p p_t (x,y) d\mu(x) d\mu(y) \right)^{1/p}
\end{equation}
and define the heat semigroup-based Besov class by
\[
\mathbf{B}^{p,\beta}(X):=\{ f \in L^p(X,\mu)\, :\, \| f \|_{p,\beta} <+\infty \}.
\]

The following result is essentially proved in \cite{MR2574998}, see also  \cite[Theorem 2.4]{BV3} and its proof. It establishes the basic and fundamental connection between the Besov-Lipschitz spaces and the study of heat kernels.

\begin{theorem}\label{Besov characterization}
Let $p \ge 1$ and $\alpha > 0$. We have $\mathcal{B}^{\alpha,p}(X) = \mathbf{B}^{p,\frac{\alpha}{d_w}}(X)$ and 
there exist constants $c,C>0$ such that for every $f \in \mathcal{B}^{\alpha,p}(X)$ and $r \in (0,\mathrm{diam}(X)) $,
\[
c \sup_{s\in(0,\mathrm{diam}(X))}E_{p,\alpha}(f,s)^{1/p} \le \| f \|_{p,\alpha/d_w} 
\le C  \left( \sup_{s\in(0,r]}E_{p,\alpha}(f,s)^{1/p}+\frac{1}{r^{\alpha}} \|f\|_{L^{p}(X,\mu)} \right).
\]
In particular, if $\mathrm{diam}(X)=+\infty$, then $ \| f \|_{p,\alpha/d_w} \simeq \sup_{s>0 } E_{p,\alpha}(f,s)^{1/p}$.
\end{theorem}

\subsection{$L^p$ critical Besov exponents in strongly recurrent metric measure Dirichlet spaces}

\begin{defn}
The metric measure Dirichlet space $(X,d,\mu,\mathcal E, \mathcal F)$ with  heat kernel estimates \eqref{eq:subGauss-upper} is called strongly recurrent if  $(X,d)$ is complete, geodesic and $d_h<d_w$.
\end{defn}

Strongly recurrent metric measure Dirichlet spaces and their potential theory were extensively studied by Barlow in \cite{MR1668115}. The terminology comes from the fact that the  Hunt process associated to the heat kernel is strongly recurrent, i.e. visits a given point with probability one.  Nested fractals like the  Vicsek set or the Sierpi\'nski  gasket  are examples of strongly recurrent metric measure Dirichlet spaces. The Sierpi\'nski  carpet is also a strongly recurrent metric measure Dirichlet space.  A key property of the heat semigroup on strongly recurrent metric measure Dirichlet spaces is the following Borel to H\"older uniform regularization property for the heat semigroup that was proved in \cite{BV3}: For every $g \in L^\infty(X,\mu)$ a continuous version of $P_tg$ exists and there exists a constant $C>0$ (independent from $g$) such that for every $x,y \in X$ and $t>0$, 
\begin{align}\label{wBE}
|P_t g (x)-P_t g(y)| \le C \left( \frac{d(x,y)}{t^{1/d_w}}\right)^{d_w-d_h} \| g \|_{L^\infty(X,\mu)}.
\end{align}
Such an estimate is called the weak Bakry-\'Emery estimate in \cite{BV3}. The functional inequality \eqref{wBE} plays the same role in the fractional space setting as the estimate
\begin{align}\label{wBE2}
\| \nabla P_t g \|_{\infty} \le \frac{C}{\sqrt{t}} \| g \|_{\infty}
\end{align}
does in the Riemannian setting. The importance of \eqref{wBE2} in the study of BV functions and isoperimetric estimates has been recognized in several works of Ledoux \cite{MR1186991}, \cite{MR1600888}. Note that for Riemannian manifolds with non-negative Ricci curvature the inequality \eqref{wBE2} is a simple byproduct of the Bakry-\'Emery machinery. For fractional spaces the proof of \eqref{wBE} we gave in \cite{BV3} relies on potential theoretical results proved in \cite{MR1668115} by using probabilistic methods. It would be very interesting to have a direct analytic proof of \eqref{wBE}  which does not rely on probabilistic methods.  Using \eqref{wBE} we can prove several results concerning the Korevaar-Schoen-Sobolev spaces and we mention a few below.  First, let us introduce some terminology. Let $E\subset X$ be a Borel set. We say $x$ is a Lebesgue density point of $E$ and write $x\in E^*$ if
\begin{equation*}
	\limsup_{r\to 0^+}\frac{\mu(B(x,r)\cap E)}{\mu(B(x,r))}>0.
	\end{equation*}
The measure-theoretic boundary is $\partial^*E=E^*\cap(E^c)^*$.  Now for $r>0$  define the measure-theoretic $r$-neighborhood $\partial_r^*E$ by  
\begin{equation}\label{eq:measure r neighborhood}
\partial^*_rE:=\big(E^*\cap(E^c)_r\big)\cup \big((E^c)^*\cap E_r\big),
\end{equation}
where $E_r=\{x\in X:\,\mu(B(x,r)\cap E)>0\}$ and similarly for $(E^c)_r$. Notice that $\partial^*E\subset\cap_{r>0}\partial^*_r E\subset\partial E$, where this last is the topological boundary.  

\begin{theorem}\cite[Theorem 5.1]{BV3}
Assume that $(X,d,\mu,\mathcal E,\mathcal F)$ is a strongly recurrent metric measure Dirichlet space and that there exists a non-empty open set $E \subset X$ such that $\mu(E) <+\infty$, $\mu(E^c)>0$, and 
\[
\limsup_{r \to 0} \frac{1}{r^{d_h}} \mu \left( \partial^*_r E\right) <+\infty
\]
where $\partial^*_r E$ is the  measure-theoretic $r$-neighborhood of $E$. Then $1_E \in \mathcal{B}^{d_h,1}(X)$ and the property $\mathcal{P}(1,d_h)$ holds.
\end{theorem}

In the setting of the previous theorem the Korevaar-Schoen-Sobolev space $KS^{d_h,1}(X)$ was interpreted in \cite{BV3} as a space of bounded variation (BV) functions and the property  $\mathcal{P}(1,d_h)$ allowed to develop a rich theory which for instance applies to any nested fractal. On the other hand, this result does not apply to infinitely ramified fractals like the Sierpi\'nski  carpet, see Figure \ref{fig-SC}.
\begin{figure}[htb]\centering
	\includegraphics[trim={60 10 180 60},height=0.25\textwidth]{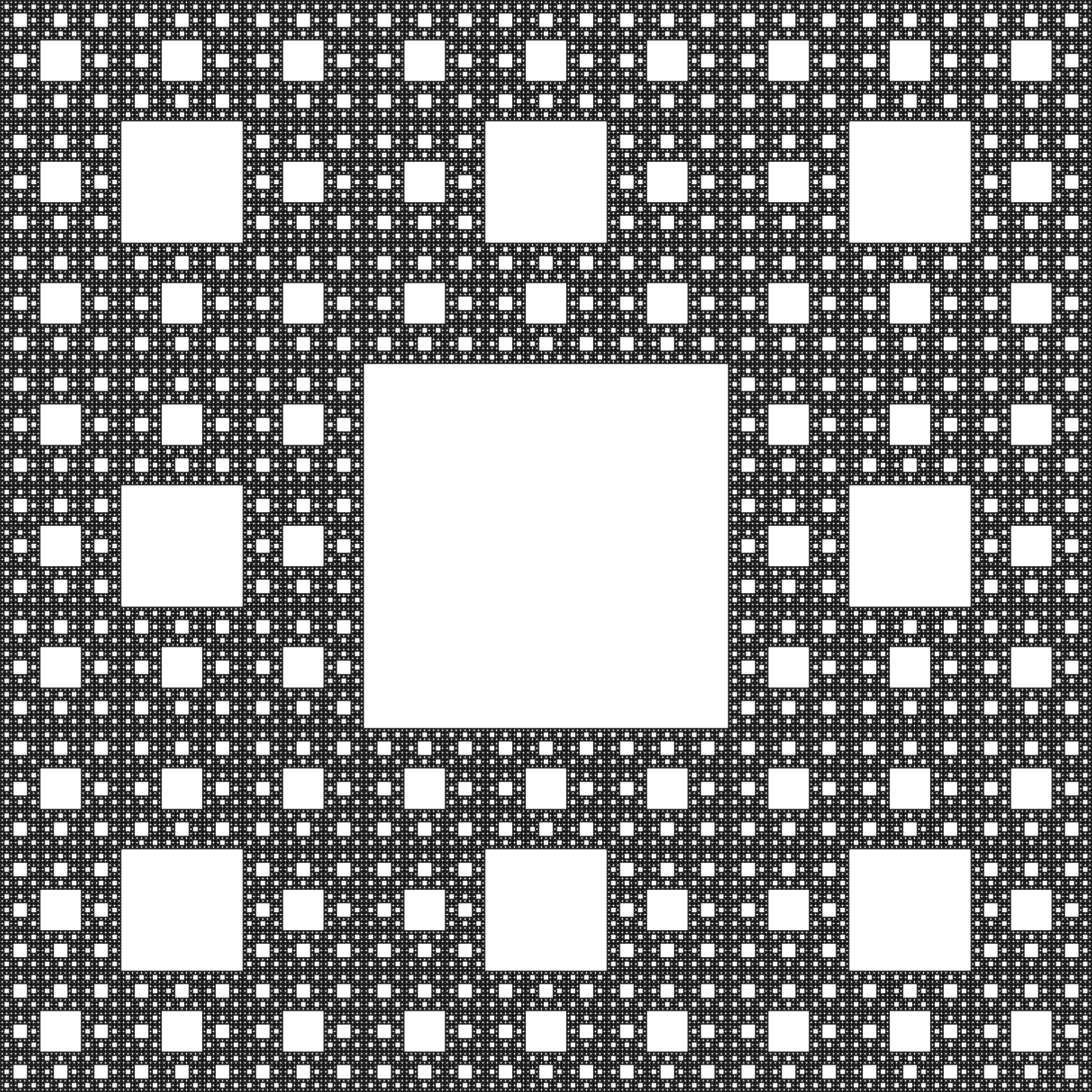}
	\caption{Sierpi\'nski carpet.} \label{fig-SC}
\end{figure}
It was conjectured in \cite{BV3} that for the Sierpi\'nski  carpet one has instead
\[
\alpha_1=2 \frac{\log 2}{\log 3}=d_h-d_{th}+1
\]
where $d_{th}$ is the topological Hausdorff dimension defined in \cite{BALKA2015881}. This conjecture is still open today as far as we know.  The next result provides some estimates on the $L^p$  critical Besov exponents of strongly recurrent metric measure Dirichlet spaces.

\begin{theorem}\cite[Theorem 3.11]{BV3}\label{estimate kappa alpha}
Assume that $(X,d,\mu,\mathcal E,\mathcal F)$ is a strongly recurrent metric measure Dirichlet space, then the $L^p$ critical Besov  exponents of $(X,d,\mu)$ satisfy:
\begin{itemize}
\item For $1 \le p \le 2$ we have $\frac{d_w}{2} \le \alpha_p \leq \Bigl(1-\frac{2}{p}\Bigr)(d_w-d_h)+\frac{d_w}{p}$. 
\item For $ p \ge 2$ we have $\Bigl(1-\frac{2}{p}\Bigr)(d_w-d_h)+\frac{d_w}{p} \le \alpha_p \le  \frac{d_w}{2}$.
\end{itemize}
\end{theorem}

We note that for the Vicsek set one has for every $p \ge 1$, $\alpha_p = \Bigl(1-\frac{2}{p}\Bigr)(d_w-d_h)+\frac{d_w}{p}$ because in that case $d_w-d_h=1$. On the other hand for the real line one has for every $p \ge 1$, $\alpha_p=\frac{d_w}{2}=1$. Therefore, in a sense, the above estimates are optimal over the range of all possible strongly recurrent metric measure Dirichlet spaces. However, for spaces satisfying $d_w-d_h<1$, like the Sierpi\'nski  gasket, the lower bound $\alpha_p \ge \Bigl(1-\frac{2}{p}\Bigr)(d_w-d_h)+\frac{d_w}{p}$, $p \ge 2$, is  interesting only when $2 \le p \le \frac{2d_h-d_w}{1-d_w+d_h}$ because we know that we always have $\alpha_p \ge 1$. When $p \ge 2$, the upper bound $\alpha_p \le \frac{d_w}{2}$ is not very good for large values of  $p$ because we know from Theorem \ref{estimate sup} that $\alpha_p\le 1+\frac{d_h}{p}$.  In view of this discussion and of the known results in the Vicsek set and the Sierpi\'nski  gasket, we can ask the following question: Is it true that for a nested fractal we have for every $p \ge 1$, $\alpha_p \le 1+\frac{d_h-1}{p}$ ?

The quantity
\[
\delta=\inf  \{ p \ge 1: p \alpha_p >d_h\}
\]
seems to be of significance. In the theory of Korevaar-Schoen-Sobolev spaces, it is the infimum of the exponent $p$ for which the space $KS^{\alpha_p ,p}(X)$ can be embedded into the space of continuous functions by using Theorem \ref{T:Morrey_local}; For instance, for the Vicsek set or the Sierpi\'nski gasket, the previous results show that $\delta=1$. Actually, from Lemma 4.10 in \cite{gao2023heat}, for every nested fractals  one has $p \alpha_p >d_h$ for every $p>1$. Therefore one  also has $\delta=1$ for every nested fractals. A similar exponent appears in Kigami's work \cite{Kigami} where it is conjectured to be equal to the Ahlfors regular conformal dimension of the space. Kigami's conjecture was recently proved in \cite{CCK}. It is natural to make the same conjecture in our setting.

%
%

\noindent
Fabrice Baudoin: \url{fbaudoin@math.au.dk}\\
Department of Mathematics,
Aarhus University

\end{document}